\documentclass[12pt,a4paper,reqno]{amsart}
\usepackage[utf8]{inputenc}
\usepackage[american]{babel}
\usepackage{amsmath,amssymb,amsthm,amscd}
\usepackage{mathrsfs}
\usepackage[active]{srcltx}
\usepackage{pgf,tikz}
\usepackage{mathrsfs}
\usepackage{enumerate}
\usepackage{braket}
\usetikzlibrary{arrows}
\usepackage{dsfont}
\usepackage{paralist}
\usepackage{comment}
\usepackage[colorlinks,linkcolor={black},citecolor={black},urlcolor={black}]{hyperref}
\usepackage[margin=2.6cm,bmargin=3cm,tmargin=3cm]{geometry}

\usepackage{enumitem, hyperref}
\makeatletter
\def\namedlabel#1#2{\begingroup
	#2%
	\def\@currentlabel{#2}%
	\phantomsection\label{#1}\endgroup
}
\makeatother

\newcommand{\Var}{\mathrm{Var}}
\newcommand{\E}{\mathbb{E}}

\newcommand{\Inf}{\operatorname{Inf}}
\newcommand{\real}{\mathbb{R}}
\newcommand{\com}{\mathbb{C}}

\newcommand{\tr}{\operatorname{tr}}

\newcommand{\cH}{\mathcal{H}}
\newcommand{\un}{\mathbf{1}}
\renewcommand{\L}{\mathcal L}

\renewcommand{\d}{d}
\renewcommand{\phi}{\varphi}
\renewcommand{\epsilon}{\varepsilon}
\newcommand{\supp}{\operatorname{supp}}

\theoremstyle{plain}
\newtheorem{theorem}{Theorem}[section]
\newtheorem{corollary}[theorem]{Corollary}
\newtheorem{lemma}[theorem]{Lemma}
\newtheorem{proposition}[theorem]{Proposition}
\newtheorem{conjecture}[theorem]{Conjecture}

\theoremstyle{remark}
\newtheorem{remark}[theorem]{Remark}

\newtheorem*{claim*}{Claim}
\newtheorem*{remark*}{Remark}
\newtheorem*{example*}{Example}
\newtheorem*{notation*}{Notation}
\numberwithin{equation}{section}

\makeatletter
\@namedef{subjclassname@2020}{\textup{2020} Mathematics Subject Classification}
\makeatother

\title[Quantum Talagrand, KKL and Friedgut's junta theorems]{Quantum Talagrand, KKL and Friedgut's theorems\\ and the learnability of quantum Boolean functions }

\author{Cambyse Rouz\'e}
\address[C. Rouz\'e]{Department of Mathematics, Technical University of Munich (TUM), Center Mathematics, M5, Boltzmannstrasse 3
	85748 Garching, Germany; Inria, Télécom Paris - LTCI, Institut Polytechnique de Paris, 91120 Palaiseau, France}
\email{rouzecambyse@gmail.com}

\author{Melchior Wirth}
\address[M. Wirth]{Institute of Science and Technology Austria (ISTA),
Am Campus 1, 3400 Klosterneuburg, Austria}
\email{melchior.wirth@ist.ac.at}

\author{Haonan Zhang}
\address[H. Zhang]{Institute of Science and Technology Austria (ISTA),
Am Campus 1, 3400 Klosterneuburg, Austria;
Department of Mathematics, University of California, Irvine, CA 92617, USA}
\address[H. Zhang]{(Current address) Department of Mathematics, University of South Carolina, Columbia, SC, 29208, USA}
\email{haonanzhangmath@gmail.com}

\subjclass[2020]{}

\begin{document}

\begin{abstract}
We extend three related results from the analysis of influences of Boolean functions to the quantum setting, namely the KKL Theorem, Friedgut's Junta Theorem and Talagrand's variance inequality for geometric influences. Our results are derived by a joint use of recently studied hypercontractivity and gradient estimates. These generic tools also allow us to derive generalizations of these results in a general von Neumann algebraic setting beyond the case of the quantum hypercube, including examples in infinite dimensions relevant to quantum information theory such as continuous variables  quantum systems. Finally, we comment on the implications of our results as regards to noncommutative extensions of isoperimetric type inequalities, quantum circuit complexity lower bounds and the learnability of quantum observables. 
\end{abstract}

\maketitle


\section{Introduction}

Let $\Omega_n=\{-1,1\}^n$ be the Boolean hypercube equipped with the uniform probability measure $\mu_n$.  For any $f:\Omega_n\to \real$, we denote by $\Var(f)=\Var_{\mu_n}(f)$ its variance, i.e. $\Var(f)=\E|f-\E f|^2$. For each $1\le j\le n$, the {\em influence} of the $j$-th variable on $f$ is given by 
\begin{equation*}
\Inf_j f:=\mathbb{E}\left[\left(\frac{f-f^{\oplus j}}2\right)^2\right],
\end{equation*}
where $f^{\oplus j}(x)=f(x^{\oplus j})$ and $x^{\oplus j}$ denotes the vector in $\Omega_n$ obtained by flipping the $j$-th variable, that is, for $x=(x_1,\dots, x_n)$, 
\begin{equation*}
x^{\oplus j}:=(x_1,\dots, x_{j-1},-x_j,x_{j+1},\dots, x_n)\,.
\end{equation*}

The notion of influences appears naturally in many contexts ranging from isoperimetric inequalities \cite{Keller12,cel12talagrand}, threshold phenomena in random graphs \cite{friedgut1996every}, cryptography  \cite{linial1993constant}, etc. For these reasons, the last three decades witnessed an extensive study of their properties, which led to many applications in theoretical computer science (hardness of approximation \cite{dinur2005hardness,haastad2001some} and learning theory \cite{o2007learning}), percolation theory \cite{benjamini1999noise}, social choice theory \cite{mossel2012quantitative,ben1985collective} to cite a few. 

Karpovsky \cite{Karpovsky76} proposed the sum of the influences (also called {\em total influence}),
	$$\Inf f:=\sum_{j=1}^n \Inf_j f,$$
		as a measure of complexity of a function $f$. This first intuition was then made rigorous in \cite{linial1993constant} and \cite{boppana1997average} where  tight circuit complexity lower bounds in terms of the total influence were derived for the complexity class $\operatorname{AC}^0$ of constant depth circuits. A simple lower bound on $\Inf f$ in terms of the variance can be derived from Poincar\'{e} inequality: For all $f\colon \Omega_n\to\mathbb R$ one has \cite[Chapter 2]{odonnell_2014}
\begin{equation}\label{ineq:classical poincare}
\Var(f)\le \Inf f\,.
\end{equation}

Functions on the hypercubes $\Omega_n$ that take only values in $\{-1,1\}$ are of particular interest. These are the so-called {\em Boolean functions} and play important roles in social science, combinatorics, computer sciences and many other areas. See \cite{dewolf2008brief,odonnell_2014} for more information. Note that the $L^p$-norms, $1\le p<\infty$, of Boolean functions are always equal to 1, where the weighted $L^p$-norm of a function $f:\Omega_n\to\mathbb{R}$ is defined as
\begin{align}\label{pnorm}
	\|f\|_p:=\left(\mathbb{E}\big[|f|^p\big]\right)^{\frac{1}{p}}\,.
\end{align}
 A Boolean function $f:\Omega_n\to \{-1,1\}$ is said to be {\em balanced} if $\E f=0$. If $f$ is a Boolean function, the influence of the $j$-th variable can further be expressed as
\begin{equation*}
\Inf_j f=\mathbb P(\{x\in\Omega_n\mid f(x)\neq f(x^{\oplus j})\}).
\end{equation*}
The Poincaré inequality \eqref{ineq:classical poincare} implies that there exists $j\in \{1,\dots,n\}$ such that $\Inf_j f\geq 1/n$. Note that Poincaré inequality \eqref{ineq:classical poincare} can be tight, e.g. for balanced Boolean function $f(x)=x_1.$ So it may happen that the total influence $\approx$ variance. Is it possible that all the influences are small simultaneously, that is, $\Inf_j(f)\approx \Var(f)/n$ for all $1\le j\le n$? Quite surprisingly, the answer is negative; a celebrated result of Kahn, Kalai and Linial \cite{kkl89influence} predicts that {\em every balanced Boolean function has an influential variable}. More precisely, Kahn, Kalai and Linial \cite{kkl89influence} proved that for any balanced Boolean function $f$ on $\Omega_n$, there exists $1\le j\le n$ such that 
\begin{equation}\label{KKL}
\Inf_j f\ge \frac{C\log (n)}{n}\,,
\end{equation}
where $C>0$ is some universal constant. So some variable has an influence at least $\Omega(\log(n)/n)$, which is larger than the order $1/n$ deduced from Poincar\'e inequality.

\medskip

This theorem of Kahn, Kalai and Linial (KKL in short) plays a fundamental role in Boolean analysis. It was further strengthened by	Talagrand \cite{talagrand94russo} and Friedgut \cite{friedgut1998boolean} in different directions.

 In his celebrated paper \cite{talagrand94russo}, Talagrand proved that for all $n\ge 1$ and $f:\Omega_n\to \real$, we have for some universal $C>0$ that 
\begin{equation}\label{ineq:talagrand}
\Var(f)\le C\sum_{j=1}^{n}\frac{\|D_j f\|^2_2}{1+\log(\|D_j f\|_2/\|D_j f\|_1)},
\end{equation}
where $D_j f(x):=\frac 1 2(f(x)-f(x^{\oplus j}))$. Note that if $f$ is Boolean, then $D_j f$ takes values only in $\{-1,0,1\}$, so that $\|D_j f \|_1=\|D_j f\|_2^2= \Inf_j f$. Therefore, this inequality of Talagrand \eqref{ineq:talagrand}, as an improvement of Poincar\'e inequality \eqref{ineq:classical poincare}, immediately implies the result of KKL. There are plenty of extensions of Talagrand's inequality \eqref{ineq:talagrand} \cite{o2013kkl,o2013sharpness,cel12talagrand}, which has become a central tool in theoretical computer science \cite{odonnell_2014}. Moreover, it provides a powerful tool to study sub-diffusive and super-concentration phenomena \cite{BenjaminiKalaiSchramm03,benjamini1999noise,chatterjee2014superconcentration,garban2014noise,auffinger201750,sosoe2018fluctuations,tanguy2018talagrand} ubiquitous to many models studied in modern probability theory (percolation, random matrices, spin glasses, etc.); see the review articles \cite{cel12talagrand,ledoux2019four} and references therein for more details.

Also related to the KKL theorem, Friedgut's Junta theorem \cite{friedgut1998boolean} states that a Boolean function with a bounded total influence essentially depends on few coordinates. More precisely, a Boolean function $f:\Omega_n\to \{-1,1\}$ is called a $k$-\textit{junta}, for $k\in \{1,\dots,n\}$ independent of $n$, if it depends on at most $k$ coordinates. When $k=1$, the function is called a \textit{dictatorship}. If $f$ is a junta, it is an immediate consequence that the total influence does not depend on $n$, i.e. $\Inf f=\mathcal{O}(1)$. Friedgut's Junta theorem provides the following converse statement: for any Boolean function $f:\Omega_n\to \{-1,1\}$ and $\epsilon>0$, there exists a $k$-junta $g:\Omega_n\to \{-1,1\}$ such that
\begin{align}\label{FriedgutJunta}
	\|f-g\|_2\le \epsilon\,,\qquad \text{ with }\quad k=2^{\mathcal{O}(\Inf f /\epsilon)}\,.
\end{align}
Since its discovery, Friedgut's Junta theorem has found many applications in random graph theory and the learnability of monotone Boolean functions \cite{o2007learning}.

\medskip

Judging from the range of applicability of these results, it is natural to consider their extensions to noncommutative or quantum settings. Partial results in this direction were obtained by Montanaro and Osborne \cite{mo10quantumboolean}. There, Boolean functions on the hypercube $\Omega_n$ were replaced by \textit{quantum Boolean functions on $n$ qubits}, that is, operators $A\in M_2(\com)^{\otimes n}$ acting on the $n$-fold tensor product of $\mathbb{C}^2$ with the additional conditions that $A=A^*$ and $A^2=\un$. Here and in what follows, $M_k(\com)$ denotes the $k$-by-$k$ complex matrix algebra. Then, the $L^2$-influence of $A$ in $j$-th coordinate is defined as $\Inf_j^2 A:= \|d_j A\|_2^2$, where we used $d_j$ to denote the quantum analogue of the bit-flip map 
$$d_j:=\mathbb I^{\otimes(j-1)}\otimes \left(\mathbb I-\frac 1 2\mathrm{tr}\right)\otimes \mathbb I^{\otimes(n-j)}$$
with $\mathbb{ I}$ being the identity map over $M_2(\com)$, and replaced the normalized $L^p$-norm on $\Omega_n$ by the normalized Schatten-$p$ norm on $M_2(\com)^{\otimes n}$. The quantum influence has already found interesting applications to quantum complexity theory \cite{bu2022complexity}. In this framework, Montanaro and Osborne \cite[Proposition 11.1]{mo10quantumboolean} proved a quantum analogue of Talagrand's inequality \eqref{ineq:talagrand}. However, this does not yield a quantum KKL as in the classical setting since we do not have the identity $\|d_j A\|_1=\|d_j A\|^2_2$ for general quantum Boolean functions. In the worst case, we may even have $\|d_j A\|_1=\|d_j A\|_2$ ({\em $j$ is a bad influence} according to \cite[Definition 11.2]{mo10quantumboolean}) and thus \eqref{ineq:talagrand} will not help anymore. For this reason, the problem of whether every balanced quantum Boolean function has an influential variable still remains open; see \cite{mo10quantumboolean} for some partial results and more discussions.

  In fact, the observation that $\|d_j A\|_1\ne \|d_j A\|^2_2$ is not exclusive to the quantum setting, and also arises for instance when considering extensions of the setup of Boolean functions on the hypercubes to functions on smooth manifolds, after replacing the uniform distribution on $\Omega_n$ by an appropriate finite measure, and the discrete derivatives $D_j$ by the partial derivatives associated to the differential structure of the manifold. In this setting, analogues of the previous results were recently obtained for the $L^1$-influences $\Inf^1_j A:=\|d_j A\|_1$, which is sometimes called {\em geometric influence} for its relation to isoperimetric inequalities \cite{Keller12,cel12talagrand,austin2016failure,bouyrie2017unified}.
 
 \medskip

In this paper, we propose to take the above considerations as a starting point for establishing quantum analogues of \eqref{ineq:classical poincare}, \eqref{KKL}, \eqref{ineq:talagrand} and \eqref{FriedgutJunta} based on the $L^1$-influences. Our first main result (Theorems \ref{quantumTalagrand} and \ref{thm:quantum Talagrand improved}) states that for any self-adjoint operator $A$ on $n$ qubits with $\|A\|\le 1$ we have 
\begin{equation}\label{ineq:kkl_quantum_1}
\|A-2^{-n}\tr(A)\|_2^2=:\Var(A)\le C \sum_{j=1}^{n} \frac{\|d_j A\|_1(1+\|d_j A\|_1)}{1+\log^+(1/\|d_j A\|_1)}
\end{equation}
for some universal $C>0$, where $\log^+$ refers to the positive part of the logarithm. In particular, this suggests that {\em every balanced quantum Boolean function has a variable that has geometric influence at least of the order $\log (n) /n$}. We also prove a quantum $L^1$-Poincar\'e inequality (Theorem \ref{thm:poincar}): for any operator $A$ on $n$ qubits we have
\begin{equation}\label{ineq:quantum_l1_poincare_intro}
\|A-2^{-n}\tr(A)\|_1\le \sum_{j=1}^{n}\Inf^1_j A\,.
\end{equation}
Therefore our result provides an alternative answer to the quantum KKL conjecture \cite[Conjecture 3 of Section 12]{mo10quantumboolean} in terms of geometric influences (Theorem \ref{thm:quantum kkl}). The inequality \eqref{ineq:kkl_quantum_1} is inspired by some results in the classical setting; see for example \cite{Keller12,cel12talagrand}. Since \eqref{ineq:kkl_quantum_1} will be our main focus, rather than \eqref{ineq:talagrand}, to distinguish them in the sequel, we shall refer to \eqref{ineq:kkl_quantum_1} as ($L^1$-)Talagrand's inequality, and \eqref{ineq:talagrand} as {\em Talagrand's $L^1$-$L^2$ variance inequality} as did in \cite{ledoux2019four}. We also have a qubit isoperimetric type inequality and a stronger form of $L^1$-Poincar\'e \eqref{ineq:quantum_l1_poincare_intro}; see Section \ref{subsect:isoperimetry} below.

Our second main result is a quantum analogue of Friedgut's Junta theorem (Theorem \ref{thm:L1friedgut} and Corollary \ref{coro:quantbooleanjunta}): for any quantum Boolean function $A\in M_2(\com)^{\otimes n}$ and $\epsilon>0$ there exists another quantum Boolean function $B\in M_2(\com)^{\otimes n}$ that is supported on $k$ subsystems such that 
\begin{equation}\label{quantumFried}
\|A-B\|_2\le  \epsilon \qquad \textnormal{ with }\qquad	k\leq 2^{\frac{270\Inf^2(A)}{\epsilon^2}}\frac{\Inf^1(A)^6}{\Inf^2(A)^5},
\end{equation}
where $\Inf^p(A):=\sum_{j=1}^n\,\Inf_j^p(A)$ with $\Inf^p_j(A)=\lVert d_j A\rVert_p^p$.

The proofs of Equations (\ref{ineq:kkl_quantum_1}) and (\ref{quantumFried}) make use of recent noncommutative generalizations of hypercontractive inequalities and gradient estimates \cite{olkiewicz1999hypercontractivity,mo10quantumboolean,kastoryano2013quantum,temme2014hypercontractivity,carlen2017gradient,datta2020relating,beigi2020quantum,WZ20,gao2021ricci,beigi2021improved}. Moreover, the generality of these tools also allows us to further extend most of our results  to the abstract von Neumann algebraic setting which contains both our previously stated results and their classical analogues previously established in \cite{cel12talagrand,bouyrie2017unified}, but also other extensions arising in noncommutative analysis and quantum information with discrete and continuous variables. As for their classical analogues, we expect our results to find many new applications to quantum information and quantum computation.

\smallskip

The rest of the paper is organized as follows: in Section \ref{qBoolean}, we recall useful definitions and results from the Fourier analysis on the quantum Boolean hypercubes including Poincar\'{e} inequality, hypercontractivity, intertwining and gradient estimates. Section \ref{sec:mainresults} is devoted to the statement and proof of our main results, namely a quantum $L^1$-Poincar\'e inequality (Theorem \ref{thm:poincar}), quantum Talagrand inequality (Theorem \ref{quantumTalagrand}), and quantum KKL theorem (Theorem \ref{thm:quantum kkl}) and a quantum Friedgut's Junta theorem (Theorem \ref{thm:L1friedgut} and Corollary \ref{coro:quantbooleanjunta}). These results are then extended to the general von Neumann algebraic setting in Section \ref{sec:generalization}.  Finally, examples and applications to quantum circuit complexity and quantum learning theory are provided in Sections \ref{sec:examples} and \ref{sec:applications}.

\subsection*{Acknowledgement} H.Z. is supported by the Lise Meitner fellowship, Austrian Science Fund (FWF) M3337. H.Z. would like to thank the American Institute of Mathematics and the AIM workshop {\em Analysis on the hypercube with applications to quantum computing}. He is also grateful to the organizers and other participants for creating an active atmosphere. The research of  C.R. has been supported by ANR project QTraj (ANR-20-CE40-0024-01) of the French
National Research Agency (ANR). C.R. acknowledges the support of the Munich Center for
Quantum Sciences and Technology, as well as the Humboldt Foundation. C.R.~would like to thank Amanda Young for fruitful discussion on the applications of Friedgut's Junta theorem to learning quantum dynamics.  The research of M.W. was funded by the Austrian Science Fund (FWF)
under the Esprit Programme [ESP 156]. For the purpose of Open Access, the authors have applied a CC BY public copyright licence to any Author Accepted Manuscript (AAM) version arising from this submission.
	The authors want to thank Francisco Escudero Gutierrez and Hsin-Yuan Huang for helpful comments on an earlier version of the paper. They are grateful to the referees for the careful reading and helpful comments. 

\section{Quantum Boolean analysis}\label{qBoolean}

Let us start by recapitulating the framework of quantum Boolean functions from \cite{mo10quantumboolean}. As a quantum analogue of functions on the Boolean hypercubes, i.e., functions of $n$ bits, we will take observables on $n$ qubits. In other words, our algebra of observables is $ M_2(\com)^{\otimes n}\cong M_{2^n}(\com)$ endowed with the operator norm $\|\cdot\|$. In what follows, we denote by $\tr$ the trace in $M_2(\com)^{\otimes n}$, and by $\tr_T$ the partial trace with respect to any subset $T$ of qubits. Following \cite[Definition 3.1]{mo10quantumboolean}, we say $A\in M_2(\com)^{\otimes n}$ is a \emph{quantum Boolean function} if $A=A^\ast$ and $A^2=\un$. Here and in what follows, $\un$ always denotes the identity operator. A quantum Boolean function $A$ is {\em balanced} if $\tr(A)=0$.

One pillar of analysis on the Boolean hypercube is that every function $f:\Omega_n\to\real$ has the Fourier--Walsh expansion, i.e. can be expressed as a linear combination of characters. Our quantum analogues of the characters for $1$ qubit are the Pauli matrices
\begin{equation*}
\sigma_0=\begin{pmatrix}1&0\\0&1\end{pmatrix},\quad \sigma_1=\begin{pmatrix}0&1\\1&0\end{pmatrix},\quad \sigma_2=\begin{pmatrix}0&-i\\i&0\end{pmatrix},\quad \sigma_3=\begin{pmatrix}1&0\\0&-1\end{pmatrix}.
\end{equation*}
Clearly, these are quantum Boolean functions, and they form a basis of $M_2(\com)$. For $s=(s_1,\dots, s_n)\in\{0,1,2,3\}^n$, we put
\begin{equation*}
\sigma_s:=\sigma_{s_1}\otimes\dots\otimes \sigma_{s_n}\, .
\end{equation*}
These are again quantum Boolean functions, and form a basis of $M_2(\com)^{\otimes n}$. Accordingly, every $A\in M_2(\com)^{\otimes n}$ can uniquely be expressed as
\begin{equation}\label{eq:quantum_fourier_expansion}
A=\sum_{s\in\{0,1,2,3\}^n} \widehat A_s \,\sigma_s
\end{equation}
where $\widehat A_s\in\mathbb C$ is the Fourier coefficient. Given $s\in\{0,1,2,3\}^n$, we call the set of indices $j$ such that $s_j\ne 0$ {\em the support of $s$}, and denote it by $\operatorname{supp}(s)$. Its cardinality is denoted by $|\operatorname{supp}(s)|$. Similarly, the support of $A$ is defined by
\begin{align}
	\operatorname{supp}(A):=\bigcup_{s|\widehat{A}_s\ne 0}\operatorname{supp}(s)\,,
\end{align}
and its cardinality is denoted by $|\operatorname{supp}(A)|$. In analogy with the classical setting, an arbitrary operator $A\in M_2(\mathbb{C})^{\otimes n}$ is called a {\em $k$-junta} if $|\operatorname{supp}(A)|\le k$. As the Pauli matrices are orthonormal with respect to the normalized Hilbert--Schmidt inner product, the coefficients $\widehat A_s$ can be recovered by
\begin{equation*}
\widehat A_s=\frac 1{2^n}\,\mathrm{tr}(\sigma_s A)\,.
\end{equation*}
Note that whenever $A$ is self-adjoint, the coefficients $\widehat{A}_s$ must be real. The quantum analogue of the bit-flip map is given by
\begin{equation*}
d_j(A):= \mathbb I^{\otimes(j-1)}\otimes \left(\mathbb I-\frac 1 2\mathrm{tr}\right)\otimes \mathbb I^{\otimes (n-j)}(A)
=\sum_{\substack{s\in \{0,1,2,3\}^n\\s_j\neq 0}}\widehat A_s \sigma_s\,.
\end{equation*}
Here $\mathbb I$ denotes the identity map on $M_2(\com)$. Note that $\mathcal L_0:=\mathbb I-\frac 1 2\mathrm{tr}$ satisfies $\mathcal L_0^2=\mathcal L_0$, so that $d_j^2=d_j$.

For $p\ge 1$, we denote by $\Inf^p_j(A):=\|d_j A\|_p^p$ the $L^p$-\textit{influence} of $j$ on the operator $A\in M_2(\com)^{\otimes n}$, and by $\Inf^p (A):=\sum_{j=1}^{n}\Inf_j^p(A)$ the associated \textit{total $L^p$-influence}, where the normalized Schatten-$p$ norm of an operator $A\in M_2(\com)^{\otimes n}$ is defined as ($|A|:=(A^\ast A)^{1/2}$)
 \begin{align*}
 	\|A\|_p:=\Big(\frac{1}{2^n}\tr\big|A\big|^p\Big)^{\frac{1}{p}}\,.
 \end{align*}
The $L^1$-influence is also called {\em the geometric influence}. For the $L^2$-influence we have
\begin{equation}\label{eq:l2_influence_dirichlet}
\Inf^2(A)=\sum_{j=1}^n \frac 1 {2^n}\mathrm{tr}((d_j A)^\ast d_j(A))=\frac 1 {2^n}\mathrm{tr}(A^\ast \mathcal L(A))
\end{equation}
with $\mathcal L:=\sum_{j=1}^n \d_j$. The operator $\mathcal{L}$ is the generator of the tensor product of the {\em quantum depolarizing semigroups}  $(P_t)_{t\ge 0}$ for the individual qubits:
\begin{equation}\label{eq:defn of semigroup}
P_t=e^{-t\mathcal L}=\left(e^{-t}\,\mathbb I+(1-e^{-t})\frac{\mathrm{tr}}{2}(\cdot)\un\right)^{\otimes n}\underset{t\to\infty}{\longrightarrow }\frac{1}{2^n}\,\tr(\cdot)\,.
\end{equation}
It is a {\em tracially symmetric quantum Markov semigroup}, whose general properties are discussed in Section \ref{sec:generalization}.

In the Fourier decomposition, we have the following convenient expressions for the $L^2$-influence:
\begin{equation}\label{eq:influence_fourier_basis}
\Inf^2(A)=\sum_{s\in \{0,1,2,3\}^n}\lvert\supp(s)\rvert \lvert\widehat A_s\rvert^2,
\end{equation}
and the semigroup $P_t$:
\begin{equation}\label{eq:QMS_fourier_basis}
P_t(A)=\sum_{s\in \{0,1,2,3\}^n}e^{-t\lvert\supp(s)\rvert}\widehat A_s\sigma_s.
\end{equation}

In addition we need the following further facts:

\begin{lemma}[Poincaré inequality, see Proposition 10.9 of \cite{mo10quantumboolean}]\label{lem:l2poincare-quantum}
For all $A\in M_2(\com)^{\otimes n}$ such that $\tr(A)=0$ and $t\geq 0$, one has 
\begin{equation*}
\lVert P_t(A)\rVert_2\leq e^{-t}\lVert A\rVert_2. 
\end{equation*}
\end{lemma}
This inequality is equivalent to
\begin{equation*}
\Var(A)\leq \frac 1{1-e^{-2t}}(\lVert A\rVert_2^2-\lVert P_t(A)\rVert_2^2). 
\end{equation*}
and is also equivalent to
\begin{equation*}
\Var(A)\leq \Inf^2(A).
\end{equation*}

\begin{lemma}[Hypercontractivity, see Theorem 8.4 of \cite{mo10quantumboolean}]\label{lem:hypercontractivity}
For all $A\in M_2(\com)^{\otimes n}$, $t\ge 0$ and $p=p(t)= 1+e^{-2t}$ one has 
\begin{equation*}
\lVert P_t(A)\rVert_2\leq \lVert A\rVert_p\,.
\end{equation*}
\end{lemma}

\begin{lemma}[Intertwining]\label{lem:intertwining}
For all $j\in\{1,\dots,n\}$ and $t\geq 0$ one has
\begin{equation*}
d_j P_t=P_t d_j.
\end{equation*}
\end{lemma}
\begin{proof}
	Follows easily from the definitions of $P_t$ and $d_j$. 
\end{proof}
We denote by $\Gamma: M_2(\com)^{\otimes n}\to M_2(\com)^{\otimes n}$ the {\em carr\'e du champ} operator associated to $P_t=e^{-t\mathcal L}$ which is defined via:
$$2\Gamma(A):=\mathcal L(A^\ast)A+A^\ast \mathcal L(A)-\mathcal L(A^\ast A)
=\sum_{j=1}^{n}\d_j(A^\ast)A+A^\ast \d_j A-\d_j(A^\ast A)\,.$$

\begin{lemma}[Gradient estimate \cite{junge2015noncommutative,WZ20}]\label{gradient}
For any $A\in M_2(\com)^{\otimes n}$ and all $t\ge 0$, 
\begin{equation*}
	\Gamma(P_t A)\le e^{-t}P_t\Gamma(A)\,.
\end{equation*}
\end{lemma}

We close this section by remarking that classical Boolean functions are special quantum Boolean functions. In fact, the Fourier--Walsh expansions of classical Boolean functions correspond to \eqref{eq:quantum_fourier_expansion} when restricting $s\in\{0,3\}^n$ .

\section{Main results for quantum Boolean functions} \label{sec:mainresults}
In this section we state and prove our main results in the restricted setting of  the quantum Boolean cube.

\subsection{A quantum $L^1$-Poincaré inequality}

We start with the following $L^1$-Poincar\'e type inequality; see also \cite{de2022limitations} for variations of this inequality and Section \ref{subsect:isoperimetry} for a stronger form. 

\begin{theorem}\label{thm:poincar}
	For all $A\in M_2(\com)^{\otimes n}$, one has
	\begin{equation}\label{ineq:l1poincare}
		\|A-2^{-n}\tr(A)\|_1\le \Inf^1 (A)\,.
	\end{equation}
\end{theorem}

\begin{proof}
	This follows from a simple use of the triangle inequality for the $L^1$-norm as well as monotonicity under the normalized partial trace. 
\end{proof}

\subsection{A quantum $L^1$-Talagrand inequality}
We first prove a quantum $L^1$-Talagrand inequality on quantum Boolean cubes that can be extended to more general von Neumann algebras; see Section \ref{sec:generalization}. We will see later that on quantum Boolean cubes the estimates can be improved, so that we may deduce a sharp quantum KKL theorem for $L^1$-influences.

\begin{theorem}\label{quantumTalagrand}
For all $A\in M_2(\com)^{\otimes n}$ with $\lVert A\rVert\leq 1$ one has
\begin{equation}\label{eq:L1Talagrand}
\Var(A)\le C \sum_{j=1}^{n} \frac{\|d_j A\|_1(1+\|d_j A\|_1)}{[1+\log^+(1/\|d_j A\|_1)]^{1/2}}\,,
\end{equation}
for some universal $C>0$.
\end{theorem}
\begin{proof}
Differentiating the function $t\mapsto \lVert P_t(A)\rVert_2^2$ one gets
\begin{align*}
\lVert A\rVert_2^2-\lVert P_T(A)\rVert_2^2=2\int_0^T \sum_{j=1}^n \lVert d_j P_t A\rVert_2^2\,dt=4\int_0^{T/2}\sum_{j=1}^n \lVert d_j P_{2t}A\rVert_2^2\,dt.
\end{align*}
By intertwining (Lemma \ref{lem:intertwining}) and hypercontractivity (Lemma \ref{lem:hypercontractivity}),
\begin{equation*}
\lVert d_j P_{2t}A\rVert_2=\lVert P_t d_j P_t A\rVert_2\leq \lVert d_j P_t A\rVert_{p(t)}
\end{equation*}
with $p(t)=1+e^{-2t}$. By H\"older's inequality,
\begin{equation*}
\lVert d_j P_t A\rVert_{p(t)}\leq \lVert d_j P_t A\rVert_1^{1/p(t)}\lVert d_j P_t A\rVert^{1-1/p(t)}.
\end{equation*}
For the term with the $L^1$-norm we use intertwining again and $L^1$-contractivity of $(P_t)_{t\ge 0}$ to get $\lVert d_j P_t A\rVert_1\leq \lVert d_j A\rVert_1$. For the term with $\|\cdot\|$-norm we use the bound derived from Lemma \ref{lem:technical lemma} below, which gives $\lVert d_j P_t A\rVert\leq (e^t-1)^{-1/2}$. Altogether,
\begin{equation*}
\lVert d_j P_{2t}A\rVert_2\le \lVert d_j P_t A\rVert_{p(t)}\leq (e^t-1)^{\frac{1-p(t)}{2p(t)}}\lVert d_j A\rVert_1^{\frac{1}{p(t)}}.
\end{equation*}
As a consequence,
\begin{equation}\label{ineq:middlestep}
\lVert A\rVert_2^2-\lVert P_T(A)\rVert_2^2\leq 4\sum_{j=1}^n \lVert d_j A\rVert_1 \int_0^{T/2}(e^t-1)^{\frac{1-p(t)}{p(t)}}\lVert d_j A\rVert_1^{\frac{2-p(t)}{p(t)}}\,dt.
\end{equation}
 Since $e^t-1\ge t$ and $p(t)\ge 1$, we have
$$(e^t-1)^{\frac{1-p(t)}{p(t)}}\le t^{\frac{1-p(t)}{p(t)}}\,.$$
Choosing $T=1$, we further show in Lemma \ref{lem:integral_bound} below that, given $a=\lVert d_j A\rVert_1$,
\begin{equation*}
\int_0^{1/2} t^{\frac{1-p(t)}{p(t)}} a^{\frac{2-p(t)}{p(t)}}\,dt\leq C\frac{1+a}{(1+\log^+(1/a))^{1/2}}\,,
\end{equation*}
for some universal constant $C>0$. We finish by combining \eqref{ineq:middlestep} and the bound (with $T=1$)
\begin{equation*}
\Var(A)\leq \frac 1{1-e^{-2T}}(\lVert A\rVert_2^2-\lVert P_T(A)\rVert_2^2)
\end{equation*}
derived from the Poincaré inequality (Lemma \ref{lem:l2poincare-quantum}).
\end{proof}

\begin{lemma}\label{lem:integral_bound}
	There exists a universal $C>0$ such that for $\alpha>0$, $a\ge 0$, $p(t)=1+e^{-2\alpha t}$ and $0\leq r\leq \min\{1,1/2\alpha\}$, we have
	\begin{equation*}
		\int_0^r t^{-(1-1/p(t))} a^{2/p(t)-1}\,dt\leq \frac{1}{\sqrt{\alpha}}\cdot\frac{C(1+a)}{(1+\log^+(1/a))^{1/2}}\, .
	\end{equation*}
\end{lemma}

\begin{proof}
		Note that $r\le 1$ and $1\le p(t)\le 2$, we have 
		\begin{equation*}
		\int_0^r t^{-(1-1/p(t))} a^{2/p(t)-1}\,dt\le \int_0^r t^{-1/2} a^{2/p(t)-1}\,dt \, .
		\end{equation*}
		Now we estimate the right hand side in two cases. If $a\geq 1$, then $a^{2/p(t)-1}\leq a$ and thus
		\begin{equation*}
		\int_0^{r} t^{-1/2} a^{2/p(t)-1}\,dt\leq a\int_0^{r} t^{-1/2}\,dt=2a\sqrt r\le  2(1+a)\,.
		\end{equation*}
		If $a<1$, then $a^{2/p(t)-1}\leq a^{\alpha t/2}$ for $0\le t\le r$. In fact, for $t\in [0,r]\subset [0, 1/2\alpha]$, we have
		\begin{equation*}
		\frac{2}{p(t)}-1=\frac{e^{2\alpha t}-1}{e^{2\alpha t}+1}\ge \frac{\alpha t}{2} \, ,
		\end{equation*}
		which is equivalent to 
		\begin{equation*}
		\varphi(\beta):=4(e^{\beta}-1)-\beta(e^\beta +1)\ge 0\, ,
		\end{equation*}
		for $\beta\in [0,2\alpha r]\subset [0,1]$. This is clear as $\varphi(0)=0$ and $\varphi'(\beta)=(3-\beta)e^\beta-1\ge 0$ for $\beta\in [0,1]$.
		Hence
		\begin{equation*}
		\int_0^{r} t^{-1/2} a^{2/p(t)-1}\,dt\leq \int_0^{1/2\alpha} t^{-1/2} a^{\alpha t/2}\,dt
		=\sqrt{\frac{2}{\alpha x}}\int_{0}^{x/4}y^{-1/2}e^{-y}dy\, ,
		\end{equation*}
		where $x=\log(1/a)>0$. It remains to show that there exists a universal constant $C>0$ such that for all $x>0$,
		\begin{equation}\label{ineq:an integral ineq}
		\sqrt{\frac{2}{x}}\int_{0}^{x/4}y^{-1/2}e^{-y}dy\leq \frac{C}{(1+x)^{1/2}}\, .
		\end{equation}
		In fact, the function 
		$$f(x):=\sqrt{\frac{2(1+x)}{x}}\int_{0}^{x/4}y^{-1/2}e^{-y}dy\, ,$$
		is continuous on $(0,\infty)$, and we have 
		$$\lim\limits_{x\to +\infty}f(x)=\sqrt{2}\int_{0}^{\infty}y^{-1/2}e^{-y}dy=\sqrt{2\pi}\, ,$$
		and by L'Hôpital's rule
		$$\lim\limits_{x\to 0^+}f(x)=\lim\limits_{x\to 0^+}\sqrt{8}x^{1/2}(1+x)^{3/2}\cdot \left(\frac{x}{4}\right)^{-1/2}e^{-x/4}=4\sqrt{2}\, .$$
		Then \eqref{ineq:an integral ineq} follows and the proof is complete. 
	\end{proof}

It remains to prove the following technical lemma:
\begin{lemma}\label{lem:technical lemma}
	Let $n\ge 1$ and $(P_t)_{t\ge 0}$ be the quantum depolarizing semigroup on $n$-qubits defined in \eqref{eq:defn of semigroup}. Then for all $t> 0$ and all $A\in M_2(\com)^{\otimes n}$ we have
	\begin{equation}\label{ineq:estimates of djPt}
		\sum_{j=1}^{n} (d_j P_t A)^\ast d_j P_t A\leq \frac{\lVert A\rVert^2}{e^t-1}\un\,.
	\end{equation}
	In particular, for each $1\le j\le n$, 
	\begin{equation*}
		(d_j P_t A)^\ast d_j P_t A\leq \frac{\lVert A\rVert^2}{e^t-1}\un \qquad \text{ and }\qquad\|d_j P_t A\|\le \frac{\|A\|}{\sqrt{e^t-1}}\,.
	\end{equation*}
	
\end{lemma}

\begin{proof}

	By definition of $(P_t)_{t\ge 0}$:
	\begin{align*}
		\lVert A\rVert^2 \un
		&\geq P_t(A^\ast A)\\
		&\geq P_t(A^\ast A)-P_t(A)^\ast P_t(A)\\
		&\overset{(1)}{=}2\int_0^t P_s\Gamma(P_{t-s}A)\,ds\\
		&		\overset{(2)}{\ge}  2\int_0^t e^s ds \cdot \Gamma(P_t A)\\
		&	=2(e^t-1)\Gamma(P_t A)\,,
	\end{align*}
	where $(1)$ follows from  the differentiation of $s\mapsto P_s(P_{t-s}(A^*)P_{t-s}(A))$, whereas $(2)$ follows from gradient estimates Lemma \ref{gradient}. Now we claim that for all $A$ and for each $1\le j\le n$ we have 
	\begin{equation}\label{ineq:claim d_j}
		\d_j(A)^\ast A+A^\ast \d_j(A)-	\d_j(A^\ast A)\ge \d_j (A)^\ast\d_j (A)\, ,
	\end{equation}
	and thus 
	\begin{equation}\label{ineq:claim_bis}
		2\Gamma(A)\ge \sum_{j=1}^{n}\d_j (A)^\ast\d_j (A)\, .
	\end{equation}
	Let us first finish the proof of the lemma given \eqref{ineq:claim_bis}. Applying \eqref{ineq:claim_bis} to $P_t A$, we may proceed with the previous estimate as 

	\begin{align*}
		\lVert A\rVert^2 \un
		\ge 2(e^t-1)\Gamma(P_t A)
		\ge (e^t-1)\sum_{j=1}^n (d_j P_t A )^\ast d_j P_t A\,,
	\end{align*}
	which proves \eqref{ineq:estimates of djPt}. 
	
	Now it remains to show \eqref{ineq:claim d_j}. For this we decompose $A=\sum_{k}A^k_1\otimes \cdots \otimes A^k_n$ into a sum of tensor products, where $k$ in $A^k_j=A^{(k)}_j$ is just an index, then 
	\begin{align*}
		&\d_j(A)^\ast A+A^\ast \d_j(A)-\d_j(A^\ast A)\\
		=&\sum_{k,l}(A^{k}_1)^\ast A^l_1\otimes \cdots \otimes \left[\mathcal L_0(A_j^k)^\ast A_j^l+(A_j^k)^\ast \mathcal L_0(A_j^l)-\mathcal L_0((A_j^k)^\ast A_j^l)\right]
		\otimes \cdots \otimes (A^{k}_n)^\ast A^l_n
	\end{align*}
	and 
	\begin{equation*}
		\d_j (A)^\ast\d_j (A)
		=\sum_{k,l}(A^{k}_1)^\ast A^l_1\otimes \cdots \otimes [\mathcal L_0(A_j^k)^\ast \mathcal L_0(A_j^l)]
		\otimes \cdots \otimes (A^{k}_n)^\ast A^l_n.
	\end{equation*}
	For any $X,Y\in M_2(\com)$, a direct computation shows
	\begin{equation*}
		\mathcal L_0(X)^\ast Y+X^\ast \mathcal L_0(Y)-	\mathcal L_0(X^\ast Y)
		=X^\ast Y+\frac{1}{2}\tr(X^\ast Y)-\frac{1}{2}\tr(X^\ast) Y-\frac{1}{2}\tr(Y)X^\ast,
	\end{equation*}
	and 
	\begin{equation*}
		\mathcal L_0(X)^\ast \mathcal L_0(Y)
		=X^\ast Y+\frac{1}{4}\tr(X^\ast)\tr(Y)-\frac{1}{2}\tr(X^\ast) Y-\frac{1}{2}\tr(Y)X^\ast.
	\end{equation*}
	Thus 
	\begin{equation*}
		\mathcal L_0(X)^\ast Y+X^\ast \mathcal L_0(Y)
		-	\mathcal L_0(X^\ast Y)-\mathcal L_0(X)^\ast \mathcal L_0(Y)
		=\frac{1}{2}\tr(X^\ast Y)-\frac{1}{4}\tr(X^\ast)\tr(Y)\,.
	\end{equation*}
	So \eqref{ineq:claim d_j} is equivalent to 
	\begin{equation*}
		\sum_{k,l} (A^{k}_1)^\ast A^l_1\otimes \cdots \otimes \left[\frac{1}{2}\tr((A_j^k)^\ast A^l_j)-\frac{1}{4}\tr((A_j^k)^\ast)\tr(A^l_j) \right]
		\otimes \cdots\otimes  (A^{k}_n)^\ast A^l_n\ge 0\,,
	\end{equation*}
	which can be reformulated as 
	\begin{equation}\label{ineq:KS}
		T_j(A^\ast A)\ge T_j(A)^\ast T_j(A)\,,
	\end{equation}
	with $T_j:=\mathbb I ^{\otimes (j-1)}\otimes \frac{1}{2}\tr \otimes \mathbb I ^{\otimes (n-j)}$. Now \eqref{ineq:KS} follows from the Kadison--Schwarz inequality \cite[Chapter 5.2]{wolf2012quantum} and that $\frac{1}{2}\tr$ is unital completely positive (over $M_2(\com)$). This finishes the proof of the claim \eqref{ineq:claim d_j} and thus the proof of the lemma.
\end{proof}

\begin{remark}
	Following the argument in \cite[Proof of Theorem 1]{cel12talagrand}, one can prove a quantum analogue of \eqref{ineq:talagrand} using similar properties of quantum depolarizing semigroups. In fact, the proof of \eqref{ineq:talagrand} does not even require strictly positive Ricci curvature lower bounds, i.e. Lemma \ref{gradient} can be weakened. We will not discuss it here as \eqref{ineq:talagrand} is not our main focus and a quantum analogue was already obtained in \cite{mo10quantumboolean}.
\end{remark}

The quantum Talagrand inequality Theorem \ref{quantumTalagrand} implies a quantum KKL for $L^1$ influences (following a similar argument in Lemma \ref{lem:auxiliary_KKL} below): for balanced quantum Boolean $A$ on $n$-qubits,
	\begin{equation*}
	\max_{1\le j\le n}\operatorname{Inf}^1_j (A)\geq C\frac{\sqrt{\log(n)}}{n}\ .
\end{equation*}
Recall that for classical Boolean functions the sharp order is $\log(n)/n$, which can be captured by tribes functions \cite[Chapter 4]{odonnell_2014}. In fact, the order $\log(n)/n$ is also sharp for quantum Boolean functions, which can be seen from the following improved version of quantum Talagrand Theorem \ref{quantumTalagrand}:

\begin{theorem}\label{thm:quantum Talagrand improved}
		For every $p\in [1,2)$ there exists a constant $C_p>0$ such that for every $n\in\mathbb N$ and $A\in M_2(\mathbb C)^{\otimes n}$ with $\lVert A\rVert\leq 1$ one has
		\begin{equation*}
			\mathrm{Var}(A)\leq C_p\sum_{j=1}^n \frac{\lVert d_j(A)\rVert_p^p(1+\lVert d_j(A)\rVert_p^p)}{1+\log^+(1/\lVert d_j(A)\rVert_p^p)},
		\end{equation*}
		where the constant can be chosen of order $C_p\sim C/(2-p)$ as $p\nearrow 2$. In particular, for $p=1$:
			\begin{equation*}
			\mathrm{Var}(A)\leq C\sum_{j=1}^n \frac{\lVert d_j(A)\rVert_1(1+\lVert d_j(A)\rVert_1)}{1+\log^+(1/\lVert d_j(A)\rVert_1)}.
		\end{equation*}
	\end{theorem}
	\begin{proof}
		Let $T>0$ be such that $p\leq 1+e^{-2T}$. By the Poincaré inequality we have
		\begin{equation*}
			\mathrm{Var}(A)\leq \frac 1{1-e^{-2T}}\left[\lVert A\rVert_2^2-\lVert P_T(A)\rVert_2^2\right]=\frac{2}{1-e^{-2T}}\int_0^T\sum_{j=1}^n \lVert d_j(P_t(A))\rVert_2^2\,dt
		\end{equation*}
		By intertwining and hypercontractivity,
		\begin{equation*}
			\frac{2}{1-e^{-2T}}\int_0^T\sum_{j=1}^n \lVert d_j(P_t(A))\rVert_2^2\,dt\leq \frac{2}{1-e^{-2T}}\int_0^T\sum_{j=1}^n \lVert d_j(A)\rVert_{p(t)}^2\,dt
		\end{equation*}
		with $p(t)=1+e^{-2t}$.
		
		By interpolation and $\lVert d_jA\rVert\leq 2\lVert A\rVert\le 2$,
		\begin{align*}
			\frac{2}{1-e^{-2T}}\int_0^T\sum_{j=1}^n \lVert d_j(A)\rVert_{p(t)}^2\,dt&\leq \frac{2}{1-e^{-2T}}\int_0^T\sum_{j=1}^n \lVert d_j(A)\rVert_{p}^{2p/p(t)}\lVert d_j(A)\rVert^{2(1-p/p(t))}\,dt\\
			&\leq \frac{8}{1-e^{-2T}}\sum_{j=1}^n \lVert d_j(A)\rVert_p^p\int_0^T\lVert d_j(A)\rVert_p^{p(2/p(t)-1)}\,dt.
		\end{align*}
		If $\lVert d_j(A)\rVert_p^p\geq 1$, then $\lVert d_j(A)\rVert_p^{p(2/p(t)-1)}\leq \lVert d_j(A)\rVert_p^p$, so that
		\begin{align*}
			\frac{8}{1-e^{-2T}} \lVert d_j(A)\rVert_p^p\int_0^T\lVert d_j(A)\rVert_p^{p(2/p(t)-1)}\,dt
			\leq &\frac{8T}{1-e^{-2T}} \lVert d_j(A)\rVert_p^{2p}\\
			\le &\frac{8T}{1-e^{-2T}} \lVert d_j(A)\rVert_p^{p}(1+\lVert d_j(A)\rVert_p^{p}).
		\end{align*}
		If $\lVert d_j(A)\rVert_p^p< 1$, then $\lVert d_j(A)\rVert_p^{p(2/p(t)-1)}\leq \lVert d_j(A)\rVert_p^{pt/2}$  (recall that $p(2/p(t)-1)\ge pt/2$ by the proof of Lemma \ref{lem:integral_bound}) and thus
		\begin{align*}
			\frac{8}{1-e^{-2T}} \lVert d_j(A)\rVert_p^p\int_0^T\lVert d_j(A)\rVert_p^{p(2/p(t)-1)}\,dt&\leq\frac{8}{1-e^{-2T}} \lVert d_j(A)\rVert_p^p\int_0^T\lVert d_j(A)\rVert_p^{pt/2}\,dt\\
			&=\frac{16}{1-e^{-2T}} \lVert d_j(A)\rVert_p^p\frac{1-\lVert d_j(A)\rVert_p^{pT/2}}{\log 1/\lVert d_j(A)\rVert_p^p}.
		\end{align*}
We claim that 
\begin{equation}\label{ineq:claim}
\frac{1-a^{-T/2}}{\log a}\le \max\left\{\frac{3T}{2},2\right\}\frac{1+a^{-1}}{1+\log a},\qquad a\ge 1.
\end{equation}		
Applying \eqref{ineq:claim} to $a=\|d_j (A)\|_p^{-p}$, we obtain
\begin{equation*}
\frac{16}{1-e^{-2T}} \lVert d_j(A)\rVert_p^p\frac{1-\lVert d_j(A)\rVert_p^{pT/2}}{\log 1/\lVert d_j(A)\rVert_p^p}
\le \frac{16\max\left\{\frac{3T}{2},2\right\}}{1-e^{-2T}} \frac{\lVert d_j(A)\rVert_p^p(1+\lVert d_j(A)\rVert_p^p)}{1+\log^+(1/\lVert d_j(A)\rVert_p^p)}.
\end{equation*}
		 Now let us prove the claim \eqref{ineq:claim} which we divide into two cases. When $a\in [1,e]$, we have 
		\begin{equation}\label{ineq:claim case 1}
		\frac{1-a^{-T/2}}{\log a}\le \frac{T}{2},\qquad 
		\frac{1+a^{-1}}{1+\log a}\ge \frac{1}{3},
		\end{equation}
		which are nothing but 
		\begin{equation*}
		f_1(a):=\frac{T}{2}\log a-1+a^{-T/2}\ge 0,\qquad 
		f_2(a):=3a+3-(a+a\log a)\ge 0.
		\end{equation*}
		A direct computation shows that when $a\in [1,e]$
		$$
		f_1'(a)=\frac{T}{2}a^{-1-\frac{T}{2}}(a^{\frac{T}{2}}-1)\ge 0, \qquad \textnormal{thus}\qquad f_1(a)\ge  f_1(1)=0,
		$$
		and 
		$$
		f_2'(a)=1-\log a\ge 0,\qquad \textnormal{thus}\qquad f_2(a)\ge f_2(1)=5.
		$$
		This proves \eqref{ineq:claim case 1} and thus the claim when $a\in [1,e]$. When $a\ge e$, we have $2\log a\ge 1+\log a$, and 
		\begin{equation}
	\frac{1-a^{-T/2}}{\log a}\le \frac{2}{1+\log a}\le \frac{2(1+a^{-1})}{1+\log a},
\end{equation}				
which proves the claim for $a\in [1,e]$.
		
Noting that $8T\le 16\max\left\{\frac{3T}{2},2\right\}$, we thus just proved for all $e^{-2T}\ge p-1$, one has
	\begin{equation*}
	\mathrm{Var}(A)\leq \frac{16\max\left\{\frac{3T}{2},2\right\}}{1-e^{-2T}}\sum_{j=1}^n \frac{\lVert d_j(A)\rVert_p^p(1+\lVert d_j(A)\rVert_p^p)}{1+\log^+(1/\lVert d_j(A)\rVert_p^p)}.
	\end{equation*}
	
Choosing $T=-\frac 1 2 \log(p-1)$, the above inequality becomes 
\begin{equation*}
	\mathrm{Var}(A)\leq C_p \sum_{j=1}^n \frac{\lVert d_j(A)\rVert_p^p(1+\lVert d_j(A)\rVert_p^p)}{1+\log^+(1/\lVert d_j(A)\rVert_p^p)},
	\end{equation*}
with $C_p=\frac{16\max\{-\frac{3}{4}\log(p-1),2\}}{2-p}$ which is of the order $32/(2-p)$ as $p\nearrow 2$.
	\end{proof}

\begin{remark}
The above Theorem \ref{thm:quantum Talagrand improved} (when $p=1$) improves Theorem \ref{quantumTalagrand}, since it gives the right order of quantum KKL as we shall see in the next. However, Theorem \ref{quantumTalagrand} can be easily extended to more general von Neumann algebras, which will be discussed in Theorem \ref{thm:general_Talagrand}. The generalization of Theorem \ref{thm:quantum Talagrand improved} is also possible but requires additional assumption which we will not discuss in the general von Neumann algebra setting. 
\end{remark}

\subsection{A KKL theorem for quantum Boolean functions}

Our quantum KKL theorem for geometric ($L^1$-)influences follows as a simple corollary of Theorem \ref{thm:quantum Talagrand improved}. First we need an elementary lemma.

\begin{lemma}\label{lem:auxiliary_KKL}
If $n\in\mathbb N$, $a_1,\dots,a_n\ge 0$ and $c>0$ such that
\begin{equation*}
\sum_{j=1}^n \frac{a_j(1+a_j)}{1+\log^+(1/a_j)}\geq c\ ,
\end{equation*}
then
\begin{equation*}
\max_{1\leq j\leq n}a_j\geq \min\left\{\frac{c}{4},1\right\}\frac{\log n}{n}\ .
\end{equation*}
\end{lemma}
\begin{proof}
If $\max_{1\leq j\leq n}a_j\geq 1/\sqrt n$, we are done, so we can assume $a_j<1/\sqrt n\le 1$ for all $j\in \{1,\dots,n\}$. Then we have
\begin{align*}
c&\leq \sum_{j=1}^n \frac{2a_j}{1+\log (1/a_j)}\\
&\leq \frac  {2}{1+\frac 1 2 \log n}\sum_{j=1}^n a_j\\
&\leq \frac {4 n}{\log n}\max_{1\leq j\leq n} a_j\ .\qedhere
\end{align*}
\end{proof}

\begin{theorem}\label{thm:quantum kkl}
	For every $1\le p<2$, there exists a constant $C_p>0$ such that for any $n\ge 1$ and any balanced quantum Boolean function $A\in M_2(\mathbb{C})^{\otimes n}$
	\begin{equation*}
	\max_{1\le j\le n}\operatorname{Inf}^p_j (A)\geq C_p\frac{\log(n)}{n}\ .
	\end{equation*}
\end{theorem}

\begin{proof}
Since $\Var(A)=1$ for any balanced quantum Boolean function, the result follows from Theorem \ref{thm:quantum Talagrand improved} with the help of Lemma \ref{lem:auxiliary_KKL}.
	\end{proof}

All combined, we have shown that every balanced quantum Boolean function has a geometrically influential variable. In fact, suppose that $A\in M_2(\com)^{\otimes n}$ is a balanced quantum Boolean function, then
$$\|A-2^{-n}\tr(A)\|_1=\|A\|_1=\|A\|_2^2=\Var(A)=1.$$
According to $L^1$-Poincar\'e inequality \eqref{ineq:l1poincare}, 
$$ \sum_{j=1}^{n}\|d_j A\|_1\ge 1\ .$$
One may wonder if $\operatorname{Inf}_j^1(A)=\|d_j A\|_1\approx 1/n$ for all $1\le j\le n$ is possible. However, our Theorem \ref{thm:quantum kkl} for $p=1$ indicates that this is not the case. There exists $j$ such that $\operatorname{Inf}^1_j (A)\ge C\log(n)/n$ for some $C>0$.

\begin{remark}
In \cite[Conjecture 3 of Section 12]{mo10quantumboolean}, the authors have conjectured a similar KKL-type result for the quantum $L^2$-influences $\operatorname{Inf}_j^2(A)$. While this influence coincides with the $L^1$-influence $\operatorname{Inf}^1_j(A)$ when $A$ is a classical Boolean function, this is not the case in the quantum setting. Hence, this conjecture in \cite{mo10quantumboolean} remains open to the best of our knowledge.

	\end{remark}

\subsection{A Friedgut's Junta theorem for quantum Boolean functions}

We recall that a Boolean function $g:\Omega_n\to \{-1,1\}$ is called a $k$-junta if it only depends on a set of at most $k<n$ bits. In \cite{friedgut1998boolean}, Friedgut showed that for any Boolean function $f:\{-1,1\}^n\to \{-1,1\}$ and $\epsilon\in(0,1)$, $f$ is $\epsilon$-close in $2$-norm to a $2^{\mathcal{O}(\Inf^2 f/\epsilon)}$-junta, where 
\begin{align}
	\Inf^2 f=\sum_{j=1}^n\Inf^2_j f
\end{align}
denotes the total $L^2$-influence of $f$, with $\Inf_j^2(f):=\|D_j f\|_2^2$. More recently, Bouyrie \cite{bouyrie2017unified} proved an $L^1$ version of Friedgut's junta theorem, more adapted to continuous models, based on the proof techniques developed in \cite{cel12talagrand} (see also \cite{austin2016failure} for a previous account of the result upon which the proof of \cite{bouyrie2017unified} relies). The next theorem constitutes a quantum generalization of the $L^1$ Friedgut's Junta theorem; see Corollary \ref{coro:quantbooleanjunta} followed. Recall that we define $k$-juntas for operators that are not necessarily Boolean.

\begin{theorem}\label{thm:L1friedgut}
For any $A\in M_2(\mathbb C)^{\otimes n}$ and any $\epsilon>0$ small enough, there exists a $k$-junta $B\in M_2(\mathbb C)^{\otimes n}$ with $\lVert A-B\rVert_2\leq \epsilon$ and
\begin{equation*}
k\leq 2^{\frac{30\Inf^2(A)}{\epsilon^2}}\frac{\lVert A\rVert_2^4\Inf^1(A)^6}{\Inf^2(A)^5}.
\end{equation*}
Moreover, $B$ can be taken to be $2^{-\lvert T\rvert}\tr_T(A)$ for some set $T\subset \{1,\dots,n\}$ of $n-k$ qubits.
\end{theorem}
\begin{proof}
Let $d=\frac{2\Inf^2(A)}{\epsilon^2}$. If $d\leq 1$, then 
\begin{equation}\label{poincareivan}
\lVert A-2^{-\lvert T\rvert}\tr_T(A)\rVert_2^2\leq \sum_{j\in T}\lVert d_j(A)\rVert_2^2\leq \Inf^2(A)\leq \frac{\epsilon^2}{2}
\end{equation}
for any subset $T$ of $\{1,\dots,n\}$ by the (non-primitive) Poincaré inequality for the tensor product of depolarizing semigroups restricted to the subset $T$ (see \cite[Example 3.1]{bardet2017estimating}).

Let us now consider the case $d>1$. Let 
\begin{equation*}
T=\left\{j\in \{1,\dots,n\}\,\bigg\vert\, \Inf^1_j(A)\leq \frac{\Inf^2(A)^5}{\Inf^1(A)^5}\frac{2^{-15d}}{\lVert A\rVert_2^4}\right\}
\end{equation*}
and $B=2^{-\lvert T\rvert}\tr_T(A)$. The matrix $A-B$ has the Fourier decomposition
\begin{equation*}
A-B=\sum_{\substack{s\in \{0,1,2,3\}^n\\s|_T\neq 0}}\widehat A_s \sigma_s.
\end{equation*}
By Plancherel's identity,
\begin{align*}
\lVert A-B\rVert_2^2=\sum_{s|_T\neq 0}\lvert \widehat A_s\rvert^2=\sum_{\substack{s|_T\neq 0\\\lvert \supp(s)\rvert>d}}\lvert \widehat A_s\rvert^2+\sum_{\substack{s|_T\neq 0\\\lvert \supp(s)\rvert\leq d}}\lvert \widehat A_s\rvert^2.
\end{align*}
Let us treat both summands on the right side separately. For the first summand,
\begin{equation*}
\sum_{\substack{s|_T\neq 0\\\lvert \supp(s)\rvert>d}}\lvert \widehat A_s\rvert^2\leq \frac 1 d\sum_{\substack{s|_T\neq 0\\\lvert \supp(s)\rvert>d}}\lvert \supp(s)\rvert\lvert \widehat A_s\rvert^2\leq \frac 1 d\Inf^2(A)= \frac{\epsilon^2} 2,
\end{equation*}
where we used formula (\ref{eq:influence_fourier_basis}) for $\Inf^2$.

For the second summand,
\begin{align*}
\sum_{\substack{s|_T\neq 0\\\lvert \supp(s)\rvert\leq d}}\lvert \widehat A_s\rvert^2&\leq e^{2dt}\sum_{s|_T\neq 0}e^{-2t\lvert\supp(s)\rvert}\lvert \widehat A_s\rvert^2\\
&\leq e^{2dt}\sum_{j\in T}\sum_{\substack{s\in \{0,1,2,3\}^n\\s_j\neq 0}}e^{-2t\lvert \supp(s)\rvert}\lvert \widehat A_s\rvert^2\\
&=e^{2dt}\sum_{j\in T}\Inf^2_j(P_t(A))
\end{align*}
for any $t\geq 0$. Here we used (\ref{eq:QMS_fourier_basis}) for the depolarizing semigroup.

Now take $t=\log 2$. By intertwining (Lemma \ref{lem:intertwining}), hypercontractivity (Lemma \ref{lem:hypercontractivity}) and interpolation,
\begin{align*}
\Inf^2_j(P_{\log 2}(A))&=\lVert d_j(P_{\log 2}(A))\rvert_2^2\\
&=\lVert P_{\log 2}(d_j(A))\rVert_2^2\\
&\leq \lVert d_j(A)\rVert_{5/4}^2\\
&\leq \lVert d_j(A)\rVert_1^{6/5}\lVert d_j(A)\rVert_2^{4/5}
\end{align*}
Since $d_j$ is a projection, $\lVert d_j(A)\rVert_2\leq \lVert A\rVert_2$. Moreover, by definition, for $j\in T$
\begin{align*}
\lVert d_j(A)\rVert_1^{6/5}=\Inf_j^1(A)^{6/5}\leq \frac{\Inf^2(A)}{\Inf^1(A)}\frac{2^{-3d}}{\lVert A\rVert_2^{4/5}}\Inf^1_j(A).
\end{align*}
Therefore,
\begin{align*}
\sum_{\substack{s|_T\neq 0\\\lvert \supp(s)\rvert\leq d}}\lvert \widehat A_s\rvert^2&\leq 4^d\sum_{j\in T}\lVert d_j(A)\rVert_1^{6/5}\lVert d_j(A)\rVert_2^{4/5}\\
&\leq 2^{-d} \frac{\Inf^2(A)}{\Inf^1(A)}\sum_{j\in T}\Inf_j^1(A)\\
&\leq 4^{-\frac{\Inf^2(A)}{\epsilon^2}}\Inf^2(A)\\
&\leq e^{-\frac{\Inf^2(A)}{\epsilon^2}}\Inf^2(A)\\
&\leq \frac{\epsilon^2}{2},
\end{align*}
where we used the elementary inequality $x\leq e^{x/2}$ for $x\geq 0$ in the last step.

Altogether we have shown that $\lVert A-B\rVert_2^2\leq \epsilon^2$. Moreover, $B$ is a $k$-junta with $k=\lvert T^c\rvert$. Since 
\begin{equation*}
\Inf^1_j(A)\geq \frac{\Inf^2(A)^5}{\Inf^1(A)^5}\frac{2^{-15d}}{\lVert A\rVert_2^4}
\end{equation*}
for every $j\in T^c$, we have
\begin{equation*}
\Inf^1(A)\geq \sum_{j\in T^c}\Inf_j^1(A)\geq \lvert T^c\rvert \frac{\Inf^2(A)^5}{\Inf^1(A)^5}\frac{2^{-15d}}{\lVert A\rVert_2^4}.
\end{equation*}
Hence
\begin{equation*}
k=\lvert T^c\rvert\leq 2^{15 d}\frac{\lVert A\rVert_2^4\Inf^1(A)^6}{\Inf^2(A)^5}=2^{\frac{30\Inf^2(A)}{\epsilon^2}}\frac{\lVert A\rVert_2^4\Inf^1(A)^6}{\Inf^2(A)^5}.\qedhere
\end{equation*}
\end{proof}

In the next corollary we restrict ourselves to quantum Boolean functions.

\begin{corollary}\label{coro:quantbooleanjunta}
For any quantum Boolean $A\in M_2(\mathbb C)^{\otimes n}$ and any $\epsilon>0$ small enough there exists a quantum Boolean $k$-junta $C\in M_2(\mathbb C)^{\otimes n}$ with $\lVert A-C\rVert_2\leq \epsilon$ and
\begin{equation*}
k\leq 2^{\frac{270\Inf^2(A)}{\epsilon^2}}\frac{\Inf^1(A)^6}{\Inf^2(A)^5}.
\end{equation*}
\end{corollary}
\begin{proof}
By Theorem \ref{thm:L1friedgut} there exists a self-adjoint $k$-junta $B\in M_2(\mathbb C)^{\otimes n}$ such that $\lVert A-B\rVert_2\leq \epsilon$ and $k\leq 2^{\frac{30\Inf^2(A)}{\epsilon^2}}\Inf^1(A)^6/\Inf^2(A)^5$. Let us now define $C:=\operatorname{sgn}(B)$ as follows: Given the spectral decomposition $B=\sum_i\lambda_i|\psi_i\rangle\langle \psi_i|$, $\operatorname{sgn}(B)=\sum_{i}\operatorname{sgn}(\lambda_i)|\psi_i\rangle\langle\psi_i|$, where the sign function $\operatorname{sgn}$ is defined as
	\begin{align*}
		\operatorname{sgn}(x):=\left\{\begin{aligned}
			&1,\,&x> 0,\,\\
			&-1,\,&x\le 0\,.
			\end{aligned}\right.
	\end{align*}

	Note that $|\lambda +\operatorname{sgn}(\lambda)|\ge 1$. Then
	\begin{equation*}
	|\lambda-\operatorname{sgn}(\lambda)|^2\le |(\lambda-\operatorname{sgn}(\lambda))(\lambda+\operatorname{sgn}(\lambda))|^2=|\lambda^2-1|^2
	\end{equation*}
	and 
\begin{equation*}
2^n\|B-C\|_2^2=\sum_{i}|\lambda_i-\operatorname{sgn}(\lambda_i)|^2\le \sum_{i}|\lambda_i^2-1|^2=2^n \|B^2-\un\|_2^2\,.
\end{equation*}
Therefore,
\begin{align*}
\|A-C\|_2\le \|A-B\|_2+\|B-C\|_2&\le \|A-B\|_2+\|B^2-\un\|_2\\
&\overset{(1)}{\le} \epsilon+\|B^2-A^2\|_2\\
&\le \epsilon+ \|(B-A)B\|_2+\|A(B-A)\|_2\\
&\le \epsilon\big(1+\|B\|+\|A\|\big)\\
&\overset{(2)}{\le} \epsilon\big(1+2\|A\|\big)\\
&\le 3\epsilon\,.
\end{align*}
where in $(1)$ we have used that $A^2=\un$, whereas in $(2)$ we used the fact that $B=2^{-|T|}\tr_T(A)$ for some set $T$ of qubits, so that $\|B\|\le \|A\|\le 1$. Moreover, we know the size of $T^c$ from Theorem \ref{thm:L1friedgut}. The result then follows after rescaling of $\epsilon$ to $\epsilon/3$.
\end{proof}

\begin{remark}
	In the case of a classical Boolean function $f$, we know that $\Inf f\equiv \Inf^1 f=\Inf^2 f$ and the bound in Corollary \ref{coro:quantbooleanjunta}, 
	simplifies as 
	\begin{align}\label{classical}
k\le  e^{\frac{270\Inf(f)}{\epsilon^2}}\Inf(f).
	\end{align}
		We therefore recover the classical Friedgut's Junta theorem.
\end{remark}

\begin{remark}
In the classical setting, other junta-type theorems related to Fourier analysis of Boolean functions may be found in \cite{friedgut2002boolean,alon2004graph,bourgain2002distribution,khot2006nonembeddability,kindler2002noise,dinur2006fourier}. While extending these results to the present quantum setting is an interesting problem, their statements do not directly involve the notion of influence that is central to our study. This interesting direction of research will therefore be considered elsewhere.
 
\end{remark}

\section{Von Neumann algebraic generalizations}\label{sec:generalization}

In this section, we generalize the main results from the previous section to the general von Neumann algebraic setting. Apart from technical challenges that arise from the fact that the underlying Hilbert space can be infinite-dimensional and the operators involved can be unbounded, most proofs run parallel to the ones for qubits once the appropriate assumptions are identified. As demonstrated in the next section, these hypotheses are satisfied for a number of interesting examples besides the qubit systems treated in Section~\ref{sec:mainresults}.

We start recapitulating some basic von Neumann algebra theory. As a general reference, we refer to \cite{Tak02, Tak03}. Let $\cH$ be a Hilbert space and $B(\cH)$ the space of all bounded linear operators on $\cH$. The $\sigma$-weak topology on $B(\cH)$ is the topology induced by the seminorms $\lvert\tr(\,\cdot\,x)\rvert$, where $x$ runs over the set of all trace-class operators. A \emph{von Neumann algebra} $\mathcal M$ on $\cH$ is a unital $\ast$-subalgebra of $B(\cH)$ that is closed in the $\sigma$-weak topology. A linear functional on $\mathcal M$ is called \emph{normal} if it is continuous with respect to the $\sigma$-weak topology. The set of all normal linear functionals on $\mathcal M$ is denoted by $\mathcal M_\ast$, and the obvious dual pairing between $\mathcal M$ and $\mathcal M_\ast$ establishes an isometric isomorphism between $\mathcal M$ and $(\mathcal M_\ast)^\ast$.

A \emph{state} on $\mathcal M$ is a positive linear functional $\phi\colon \mathcal M\to\com$ such that $\phi(\un)=1$. A state is called \emph{faithful} if $\phi(x^\ast x)=0$ implies $x=0$. For a faithful normal state $\phi$ on $\mathcal M$ let $\cH_\phi$ denote the completion of $M$ with respect to the inner product
\begin{equation*}
	\langle\,\cdot\,,\cdot\,\rangle_\phi\colon \mathcal M\times \mathcal M\to\com,\,(x,y)\mapsto \phi(x^\ast y),
\end{equation*}
and let $\Lambda_\phi(x)$ denote the image of $x$ inside $\cH_\phi$. The GNS representation is defined by $\pi_\phi(x)\Lambda_\phi(y)=\Lambda_\phi(xy)$. The vector $\Lambda_\phi(\un)$ is a cyclic and separating vector for $\pi_\phi(\mathcal M)$, which is denoted by $\Omega_\phi$. We routinely identify $\mathcal M$ with $\pi_\phi(\mathcal M)$.

For the definition of the noncommutative $L^p$ spaces, we need some basic modular theory. The operator
\begin{equation*}
	S_0\colon \Lambda_\phi(\mathcal M)\to \Lambda_\phi(\mathcal M),\,\Lambda_\phi(x)\mapsto\Lambda_\phi(x^\ast)
\end{equation*}
is a closable anti-linear operator on $\cH_\phi$. Let $S$ denote its closure and $S=J\Delta^{1/2}$ the polar decomposition of $S$. The operator $J$ is an anti-unitary involution, called the {\em modular conjugation}, and $\Delta=S^\ast S$ is called the \emph{modular operator}.

The symmetric embedding $i_2$ of $\mathcal M$ into $\cH_\phi$ is given by $i_2(x)=\Delta^{1/4}\Lambda_\phi(x)$ and the symmetric embedding $i_1$ of $\mathcal M$ into $\mathcal M_\ast$ is uniquely determined by the relation
\begin{equation*}
	\langle i_2(x^\ast),i_2(y)\rangle=i_1(x)(y),
\end{equation*}
or in other words, $i_1=i_2^\ast J i_2$ if we view $J$ as an isomorphism between $\cH_\phi$ and $\overline \cH_\phi\cong \cH_\phi^\ast$.

Kosaki's interpolation $L^p$ spaces \cite{KOSAKI1984} are defined as the complex interpolation space
\begin{equation*}
	L^p(\mathcal M,\phi)=(\mathcal M_\ast,i_1(\mathcal M))_{1/p}
\end{equation*}
for $p\in (1,\infty)$. Thus we get embeddings $i_p\colon \mathcal M\to L^p(\mathcal M,\phi)$ for $p\in (1,\infty)$ with
\begin{equation*}
	\lVert i_p(x)\rVert\leq \lVert i_1(x)\rVert^{1/p}\lVert x\rVert^{1-1/p}
\end{equation*}
for all $x\in \mathcal M$. In particular, $L^2(\mathcal M,\phi)\cong \cH_\phi$ isometrically, and the definition of $i_2$ is consistent with the definition given before under this identification.

In the case $\mathcal M=M_n(\mathbb C)$, every state $\phi$ on $\mathcal M$ is of the form $\phi=\tr(\,\cdot\,\sigma)$ for some density matrix $\sigma$. The state $\phi$ is faithful if and only if $\sigma$ is invertible. In this case $L^p(M_n(\mathbb C),\phi)$ can be identified with $M_n(\mathbb C)$ with the norm $\tr(\lvert\cdot\rvert^p)^{1/p}$, and the embedding $i_p$ is given by $i_p(x)=\sigma^{1/2p}x\sigma^{1/2p}$. In particular, $\lVert i_p(x)\rVert=\tr(\lvert \sigma^{1/2p}x\sigma^{1/2p}\rvert^p)^{1/p}$, which is the expression for the $L^p$ norm commonly used in quantum information theory.

A \emph{quantum Markov semigroup (QMS)} on $\mathcal M$ is a family $(P_t)_{t\geq 0}$ of normal bounded linear operators on $\mathcal M$ such that
\begin{itemize}
	\item $P_0=\mathrm{id}_{\mathcal M}$, $P_s P_t=P_{s+t}$ for $s,t\geq 0$,
	\item $P_t(x)\to x$ as $t\searrow 0$ in the $\sigma$-weak topology for every $x\in \mathcal M$,
	\item $\sum_{j,k=1}^n y_j^\ast P_t(x_j^\ast x_k)y_k\geq 0$ for all $x_1,\dots,x_n,y_1,\dots,y_n\in \mathcal M$ and $t\geq 0$,
	\item $P_t(\un)=1$ for all $t\geq 0$.
\end{itemize}
If $(P_t)_{t\ge 0}$ is a quantum Markov semigroup on $\mathcal M$, then $P_t$ has a pre-adjoint $(P_t)_\ast\colon \mathcal M_\ast\to \mathcal M_\ast$ for every $t\geq 0$, and $((P_t)_\ast)_{t\geq 0}$ is a strongly continuous semigroup on $\mathcal M_\ast$. The QMS $(P_t)_{t\ge 0}$ is called \emph{KMS-symmetric with respect to $\phi$} if 
\begin{equation*}
	\langle i_2(P_t(x)),i_2(y)\rangle=\langle i_2(x),i_2(P_t(y))\rangle
\end{equation*}
for all $x,y\in \mathcal M$ and $t\geq 0$. In this case, for all $p\in [1,\infty)$ and $t\geq 0$ the operator 
\begin{equation*}
	i_p(\mathcal M)\to i_p(\mathcal M),\,i_p(x)\mapsto i_p(P_t(x))
\end{equation*}
extends to a contraction $P_t^{(p)}$ on $L^p(\mathcal M,\phi)$, and $(P_t^{(p)})_{t\ge 0}$ is a strongly continuous semigroup. In particular, $(P_t)_\ast=P_t^{(1)}$. Occasionally we also write $P_t^{(\infty)}$ for $P_t$.

The generator of $(P_t^{(p)})_{t\ge 0}$ is defined by
\begin{align*}
	D(\L_p)&=\{x\in L^p(\mathcal M,\phi)\mid \lim_{t\searrow 0}\frac 1 t(x-P_t^{(p)}(x))\text{ exists}\},\\
	\L_p(x)&=\lim_{t\searrow 0}\frac 1 t(x-P_t^{(p)}(x)),
\end{align*}
where the limit is taken in the norm topology if $p\in [1,\infty)$ and in the $\sigma$-weak topology if $p=\infty$. We also write $\L$ for $\L_\infty$. Note that there are differing sign conventions for the generator; with our convention, $\L_2$ is a positive self-adjoint operator on $L^2(\mathcal M,\phi)$.

We make the following assumption:
	\begin{description}[style=multiline, labelwidth=1.5cm]
			\item[\namedlabel{(A0)}{(H0)}] There exists a $\ast$-subalgebra $\mathcal A$ of $D(\L)$ which is $\sigma$-weakly dense in $\mathcal M$ and invariant under $(P_t)_{t\ge 0}$.
	\end{description}
We can then define the \emph{carré du champ} operator as follows:
\begin{equation*}
	\Gamma\colon \mathcal A\times\mathcal A\to \mathcal A,\,\Gamma(x,y)=\frac 1 2(\L(x)^\ast y+x^\ast \L(y)-\L(x^\ast y)).
\end{equation*}
We write $\Gamma(x)$ for $\Gamma(x,x)$.

We will further use the following assumption:
	\begin{description}[style=multiline, labelwidth=1.5cm]
		\item[\namedlabel{(A1)}{(H1)}] Bakry--Émery gradient estimate: There exists $K\in\mathbb R$ such that
	\begin{equation*}
		\Gamma(P_t(x))\leq e^{-2Kt}P_t(\Gamma(x))
	\end{equation*}
	for all $x\in \mathcal A$ and $t\geq 0$.
\end{description}

To avoid case distinctions, the following notation will come in handy:
\begin{equation*}
	e_K(t)=2\int_0^t e^{2Ks}\,ds=\begin{cases}
		\frac{e^{2Kt}-1}{K}&\text{if }K\neq 0,\\2t&\text{if }K=0.
	\end{cases}
\end{equation*}
Further, we write $K_-$ for the negative part of a real number $K$. The following result is an analog of Lemma \ref{lem:technical lemma}.
\begin{lemma}\label{lem:reverse_Poincare}
	If $(P_t)_{t\ge 0}$ is a QMS satisfying \ref{(A0)}--\ref{(A1)}, then
	\begin{equation*}
		\Gamma(P_t(x))\leq \frac 1{e_K(t)}(P_t(x^\ast x)-P_t(x)^\ast P_t(x))\leq \frac {\lVert x\rVert^2}{e_K(t)}
	\end{equation*}
	for all $x\in\mathcal A$ and $t\geq 0$.
	
	In particular,
	\begin{equation*}
		\Gamma(P_t(x))\leq\begin{cases} \frac{\lVert x\rVert^2}{2t}&\text{if }K\geq 0,\\ \frac{\lVert x\rVert^2}{t}&\text{if }K< 0\end{cases}
	\end{equation*}
	for all $x\in \mathcal A$ and $t\in \left[0,\frac 1{2K_-}\right)$.
\end{lemma}

	\begin{description}[style=multiline, labelwidth=1.5cm]
		\item[\namedlabel{(A2)}{(H2)}] There exists a finite family of linear self-adjoint maps $d_j\colon \mathcal A\to \mathcal M$, $j\in \mathcal J$, such that
	\begin{equation}\label{(A2-1)}
		\langle i_2(x),i_2(\L(x))\rangle=\sum_{j\in\mathcal J}\lVert i_2(d_j(x))\rVert^2\tag{H2-1}
	\end{equation}
	and a constant $M>0$ such that
	\begin{equation}\label{(A2-2)}
		\max_{j\in \mathcal J}\lVert d_j(x)\rVert\leq M\lVert \Gamma(x)\rVert^{1/2}\tag{H2-2}
	\end{equation}
	for all $x\in \mathcal A$.
\end{description}

Note that (\ref{(A2-1)}) implies in particular that the series on the right side converges for all $x\in \mathcal A$, and by polarization,
\begin{equation}\label{eq:polarization}
	\langle i_2(x),i_2(\mathcal L(y))\rangle=\sum_{j\in\mathcal J}\langle i_2(d_j(x)),i_2(d_j(y))\rangle
\end{equation}
for all $x,y\in\mathcal A$.

In this situation we define the $p$-influence of the $j$-th variable on $x$ by $\Inf^p_j(x)=\lVert i_p(d_j(x))\rVert^p$ and the total influence of $x$ by $\Inf^p(x)=\sum_{j\in\mathcal J}\Inf^p_j(x)$.

We say $(P_t)_{t\ge 0}$ is \emph{primitive} if $P_t(x)\to \phi(x)\un$ $\sigma$-weakly as $t\to\infty$ for every $x\in \mathcal M$.

\begin{theorem}[$L^1$-Poincaré inequality]
	If $(P_t)_{t\ge 0}$ is a primitive KMS-symmetric QMS on $\mathcal M$ satisfying \ref{(A0)}--\ref{(A2)} with $K>0$, then
	\begin{equation*}
		\frac{\sqrt K}{M}\lVert i_1(x-\phi(x)\un)\rVert\leq \frac { \pi}{2} \Inf^1(x)
	\end{equation*}
	for all $x\in\mathcal A$.
\end{theorem}
\begin{proof}
	By duality and $\sigma$-weak density of $\mathcal A$ in $\mathcal M$, we have
	\begin{equation*}
		\lVert i_1(x-\phi(x)\un)\rVert=\sup_{\substack{y\in\mathcal A\\\lVert y\rVert\leq 1}}\lvert i_1(x-\phi(x)\un)(y)\rvert=\sup_{\substack{y\in\mathcal A\\\lVert y\rVert\leq 1}}\lvert i_1(x)(y-\phi(y)\un)\rvert.
	\end{equation*}
	Since $(P_t)_{t\ge 0}$ is primitive,
	\begin{align*}
		\lvert i_1(x)(y-\phi(y)\un)\rvert&=\lim_{T\to\infty}\lvert i_1(x)(y-P_T(y))\rvert\\
		&=\left\lvert\int_0^\infty i_1(x)(\L P_t(y))\,dt\right\rvert\\
		&=\left\lvert\int_0^\infty \langle i_2(x^\ast),i_2(\L(P_t(y)))\rangle\,dt\right\rvert.
	\end{align*}
	Now by the consequence \eqref{eq:polarization} of (\ref{(A2-1)}),
	\begin{align*}
		\left\lvert\int_0^\infty \langle i_2(x^\ast),i_2(\L(P_t(y)))\rangle\,dt\right\rvert&=\left\lvert\int_0^\infty \sum_{j\in\mathcal J} \langle i_2(d_j(x^\ast)),i_2(d_j(P_t(y)))\rangle\,dt\right\rvert\\
		&\leq \sum_{j\in\mathcal J}\int_0^\infty\lvert i_1(d_j(x))(d_j(P_t(y)))\rvert\,dt
		\\
		&\leq \sum_{j\in\mathcal J}\Inf^1_j(x)\int_0^\infty\lVert d_j(P_t(y))\rVert\,dt
	\end{align*}
	By Lemma \ref{lem:reverse_Poincare} and (\ref{(A2-2)}),
	\begin{align*}
		\int_0^\infty\lVert d_j(P_t(y))\rVert\,dt\leq M\int_0^\infty \lVert\Gamma(P_t(y))\rVert^{1/2}\,dt\leq M\sqrt K \lVert y\rVert\int_0^\infty\frac {dt}{\sqrt{e^{2Kt}-1}}=\frac{\pi M}{2\sqrt K}\lVert y\rVert.
	\end{align*}
	All combined, we obtain the desired inequality.
\end{proof}

To prove our general noncommutative version of the $L^1$-Talagrand inequality, we need some more assumptions on $(P_t)_{t\ge 0}$, which we collect in the following:

	\begin{description}[style=multiline, labelwidth=1.5cm]
		\item[\namedlabel{(A3)}{(H3)}] Poincaré inequality: There exists a constant $\lambda>0$ such that
	\begin{equation*}
		\lambda \lVert i_2(x-\phi(x)\un)\rVert^2\leq \langle i_2(x),i_2(\L(x))\rangle
	\end{equation*}
	for all $x\in\mathcal A$.
	\item[\namedlabel{(A4)}{(H4)}] Hypercontractivity: There exists a constant $\alpha>0$ such that 
	\begin{equation*}
		\lVert i_2(P_t(x))\rVert\leq \lVert i_p(x)\rVert
	\end{equation*}
	for all $x\in\mathcal A$, $t\geq 0$ and $p=1+e^{-2\alpha t}$.
	\item[\namedlabel{(A5)}{(H5)}] Intertwining: There exists a constant $\mu\in\mathbb R$ such that
	\begin{equation*}
		\lVert i_p(d_j(P_t(x)))\rVert\leq e^{-\mu t}\lVert i_p(P_t(d_j(x)))\rVert
	\end{equation*}
	for all $x\in \mathcal A$, $j\in\mathcal J$, $p\in [1,\infty]$ and $t\geq 0$.
\end{description}

In fact, it is well-known \cite{olkiewicz1999hypercontractivity} that hypercontractivity \ref{(A4)} implies Poincar\'e inequality \ref{(A3)}. 

The proof of the following theorem follows the argument given by Cordero--Erausquin and Ledoux \cite[Theorem 6]{cel12talagrand} in the commutative case. We refer to the appendix for the details.

\begin{theorem}[$L^1$-Talagrand inequality]\label{thm:general_Talagrand}
	If $(P_t)_{t\ge 0}$ is a KMS-symmetric QMS on $\mathcal M$ satisfying \ref{(A0)}--\ref{(A5)}, then there exists a constant $C>0$ depending only on the constants $K,M,\alpha,\lambda,\mu$ such that
	\begin{equation*}
		\lVert i_2(x-\phi(x)\un)\rVert^2\leq C\sum_{j\in\mathcal J}\frac{\Inf^1_j(x)(1+\Inf^1_j(x))}{(1+\log^+(1/\Inf^1_j(x)))^{1/2}}
	\end{equation*}
	for all $x\in\mathcal A$ with $\lVert x\rVert\leq 1$.
\end{theorem}

Again following \cite{cel12talagrand}, we can also give a generalization of Talagrand's inequality \eqref{ineq:talagrand} in this setting.

\begin{theorem}
If $(P_t)$ is a KMS-symmetric QMS on $\mathcal M$ satisfying \ref{(A0)}, \ref{(A2)}--\ref{(A5)}, then
\begin{equation*}
\lVert i_2(x-\phi(x)\un)\rVert^2\leq \frac{2e^{(2\alpha-\mu)_+/2\lambda}}{\alpha (1-e^{-1})}\sum_{j\in\mathcal J}\frac{\Inf^2_j(x)}{1+\log(\sqrt{\Inf_j^2(x)}/\Inf_j^1(x))}
\end{equation*}
for all $x\in\mathcal A$.
\end{theorem}
\begin{proof}
By the Poincaré inequality \ref{(A3)}, we have
\begin{equation*}
\lVert i_2(x-\phi(x)\un)\rVert^2\leq \frac 1{1-e^{-1}}(\lVert i_2(x)\rVert^2-\lVert i_2(P_T(x))\rVert^2
\end{equation*}
for $T=1/2\lambda$.

Arguing as in the proof of Theorem \ref{quantumTalagrand}, we get
\begin{align*}
\lVert i_2(x)\rVert^2-\lVert i_2(P_T(x))\rVert=2\sum_{j\in \mathcal J}\int_0^T \lVert i_2(d_j(P_t(x)))\rVert^2\,dt.
\end{align*}
By \ref{(A4)} and \ref{(A5)}, 
\begin{align*}
\lVert i_2(d_j(P_t(x)))\rVert\leq e^{-\mu t}\lVert i_2(P_t(d_j(x)))\rVert\leq e^{-\mu t}\lVert i_{p(t)}(d_j(x))\rVert
\end{align*}
with $p(t)=1+e^{-2\alpha t}$.

After the change of variables $s=p(t)$ and application of Hölder's inequality we get
\begin{align*}
\lVert i_2(x-\phi(x)\un)\rVert^2&\leq \frac{2}{1-e^{-1}}\sum_{j\in \mathcal J}\int_0^T e^{-2\mu t}\lVert i_{p(t)}(d_j(x))\rVert^2\,dt\\
&=\frac 1{\alpha(1-e^{-1})}\sum_{j\in\mathcal J}\int_{1+e^{-\alpha/\lambda}}^2 (s-1)^{\mu/2\alpha-1}\lVert i_s(d_j(x))\rVert^2\,ds\\
&\leq \frac{e^{(2\alpha-\mu)_+/2\lambda}}{\alpha (1-e^{-1})}\sum_{j\in \mathcal J}\int_1^2 \lVert i_s(d_j(x))\rVert^2\,ds\\
&\leq \frac{e^{(2\alpha-\mu)_+/2\lambda}}{\alpha (1-e^{-1})}\sum_{j\in\mathcal J}\int_1^2 \Inf^1_j(x)^{4/s-2}\Inf^2_j(x)^{2-2/s}\,ds\\
&=\frac{e^{(2\alpha-\mu)_+/2\lambda}}{\alpha (1-e^{-1})}\sum_{j\in\mathcal J}\Inf_j^2(x)\int_1^2 (\Inf^1_j(x)/\Inf_j^2(x))^{2/s-2}\,ds.
\end{align*}
From here, the claimed inequality follows from an elementary bound on the last integral (compare \cite[Theorem 1]{cel12talagrand}).
\end{proof}


Since the $p$-influences for different $p$ do not coincide in the quantum setting, this version of Talagrand's inequality does not imply a KKL bound. However, we still have the following weaker bound as consequence of Theorem \ref{thm:general_Talagrand}. Again, the proof can be found in the appendix.

\begin{theorem}\label{ThmKKLgeneral}
If $(P_t)$ is a KMS-symmetric QMS on $\mathcal M$ satisfying \ref{(A0)}--\ref{(A5)} and the cardinality $n$ of $\mathcal J$ is finite, then there exists $C'>0$ depending only on the constants $K,L,M,\alpha,\lambda,\mu$ 
such that
\begin{equation*}
\max_{j\in\mathcal J}\Inf^1_j(x)\geq C'\frac{\sqrt{\log n}}{n}
\end{equation*}
for all self-adjoint $x\in\mathcal A$ with $\lVert i_2(x)\rVert=1$, $\lVert x\rVert\leq 1$ and $\phi(x)=0$.
\end{theorem}

\begin{remark}
	The sharpness of the bound derived in Theorem \ref{ThmKKLgeneral} in the present general context was shown in \cite{Keller12}. 
\end{remark}

To prove our generalized version of Friedgut's junta theorem, we need one last assumption on $(P_t)$. For that purpose, if $\mathcal I\subset \mathcal J$, let $E_{\mathcal I}$ denote the orthogonal projection onto $\bigcap_{i\in\mathcal I}\overline{i_2(\ker d_i)}$ in $L^2(\mathcal M,\phi)$.

	\begin{description}[style=multiline, labelwidth=1.5cm]
	\item[\namedlabel{(A6)}{(H6)}] There exists a constant $\nu>0$ such that
	\begin{equation*}
		\nu\lVert i_2(x)-E_{\mathcal I}(i_2(x))\rVert^2\leq \sum_{i\in\mathcal I}\lVert i_2(d_i(x))\rVert^2
	\end{equation*}
for every $x\in \mathcal A$ and $\mathcal I\subset\mathcal J$.
\end{description}

If $(P_t)$ is primitive, then $E_{\mathcal J}(i_2(x))=i_2(\phi(x)\un)$. Thus \ref{(A6)} is a strengthening of the Poincaré inequality from \ref{(A3)} in the case of primitive QMS.

\begin{lemma}\label{lem:cor_Poincare}
If $(P_t)$ is a KMS-symmetric QMS on $\mathcal M$ satisfying \ref{(A0)}, \ref{(A2)}, \ref{(A4)}--\ref{(A6)}, then for any $x\in \mathcal A$, $t,\eta>0$ and $\mathcal I\subset \mathcal J$ such that $\Inf^1_i(x)\leq \eta$ for all $i\in\mathcal I$ one has
\begin{equation*}
\lVert (\mathrm{id}-E_{\mathcal I})(i_2(P_t(x)))\rVert^2\leq \frac{e^{-\mu t}}{\nu}(\eta\Inf^1(x))^{q(t)}(\Inf^2(x))^{1-q(t)},
\end{equation*}
where $q(t)=\frac{1-e^{-2\alpha t}}{1+e^{-2\alpha t}}$.
\end{lemma}
\begin{proof}
By \ref{(A6)} we have
\begin{equation*}
\lVert i_2(P_t(x))-E_{\mathcal I}(i_2(P_t(x)))\rVert^2\leq \frac 1 \nu\sum_{i\in\mathcal I}\lVert i_2(d_i(P_t(x)))\rVert^2.
\end{equation*}
By \ref{(A4)}, \ref{(A5)} and interpolation,
\begin{align*}
\lVert i_2(d_i(P_t(x)))\rVert&\leq e^{-\mu t}\lVert i_2(P_t(d_i(x)))\rVert\\
&\leq  e^{-\mu t}\lVert i_{p(t)}(d_i(x))\rVert\\
&\leq e^{-\mu t}\lVert i_1(d_i(x))\rVert^{q(t)}\lVert i_2(d_i(x))\rVert^{1-q(t)}.
\end{align*}
Therefore,
\begin{align*}
\lVert (\mathrm{id}-E_{\mathcal I})(i_2(P_t(x)))\rVert^2&\leq \frac{e^{-\mu t}}{\nu} \sum_{i\in\mathcal I} \lVert i_1(d_i(x))\rVert^{2q(t)}\lVert i_2(d_i(x))\rVert^{2(1-q(t))}\\
&\leq \frac{e^{-\mu t}}{\nu}\left(\sum_{i\in\mathcal I}\lVert i_1(d_i(x))\rVert^2\right)^{q(t)}\left(\sum_{i\in\mathcal I}\lVert i_2(d_i(x))\rVert^2\right)^{1-q(t)}\\
&\leq \frac{e^{-\mu t}}{\nu}\left(\sum_{i\in\mathcal I} \eta\Inf^1_i(x)\right)^{q(t)}\left(\sum_{j\in\mathcal J}\Inf^2_j(x)\right)^{1-q(t)}\\
&=\frac{e^{-\mu t}}{\nu}(\eta \Inf^1(x))^{q(t)}(\Inf^2(x))^{1-q(t)}.\qedhere
\end{align*}
\end{proof}

\begin{lemma}\label{lem:elem_spectral_th}
If $A$ is a positive self-adjoint operator on a Hilbert space $\mathcal H$, then
\begin{equation*}
\lVert \xi-e^{-tA}\xi\rVert_{\mathcal H}^2\leq t \lVert A^{1/2}\xi\rVert_{\mathcal H}^2
\end{equation*}
for every $\xi\in D(A^{1/2})$ and $t\geq 0$.
\end{lemma}
\begin{proof}
This follows by the spectral theorem from the scalar inequality $(1-e^{-tx})^2\leq tx$ for $t,x\geq 0$. 
\end{proof}
A version of Friedgut's junta theorem in this setting now reads as follows. Again, the proof can be found in the appendix.

\begin{theorem}\label{thm:Friedgut_general_vN}
Let $(P_t)$ be a KMS-symmetric QMS on $\mathcal M$ satisfying \ref{(A0)}, \ref{(A2)}, \ref{(A4)}--\ref{(A6)}. There exists a constant $C>0$ depending only on $\alpha$ and $\nu$ such that for every $x\in \mathcal A$ and $0<\epsilon\leq 2/\nu$ there exists a set $\mathcal I\subset\mathcal J$ such that $\lVert i_2(x)-E_{\mathcal I}(i_2(x))\rVert\leq \epsilon$ and
\begin{equation*}
\lvert \mathcal J\setminus \mathcal I\rvert\leq \begin{cases}\Inf^1(x)^2 \exp\left(C\mu_-+\frac{C \Inf^2(x)}{\epsilon^2}\log \frac{2\Inf^2(x)}{\nu\epsilon}\right)&\text{if }\Inf^2(x)\geq 1,\\\frac{\Inf^1(x)^2}{\Inf^2(x)}\exp\left(C\mu_- +\frac{C \Inf^2(x)}{\epsilon^2}\log\frac{2\sqrt{\Inf^2(x)}}{\nu\epsilon}\right)&\text{otherwise},\end{cases}
\end{equation*}
where $\mu_-=-\mu$ if $\mu<0$, and $0$ otherwise. 
\end{theorem}

\section{Examples}\label{sec:examples}

\subsection{Classical case}

The results in \cite{cel12talagrand,bouyrie2017unified} fit into our framework by choosing the commutative von Neumann algebras, i.e. $(\mathcal M,\varphi)=L^\infty(X,\mu)$ with $X$ a probability measure space. 

\subsection{Generalized depolarizing semigroups}We start with a simple weighted generalization of the depolarizing semigroup, also known as \textit{generalized depolarizing}: given a full-rank state $\omega$ over $\mathbb{C}^{d}$,  
\begin{align*}
	e^{t\mathcal{L}_\omega}=\big(e^{-t}\operatorname{id}+(1-e^{-t})\,\tr(\omega \,\cdot)\un\big)^{\otimes n}\,.
\end{align*}
We verify assumptions \ref{(A0)}-\ref{(A5)} for the semigroup $(e^{t\mathcal{L}_\omega})_{t\ge 0}$ . First of all, since we are in a finite dimensional case, \ref{(A0)} is directly satisfied. \ref{(A1)} was proved in \cite{junge2015noncommutative} with $K=\frac{1}{2}$. With $d_j(x)=\mathrm{id}^{\otimes(j-1)}\otimes (\mathrm{id}-\tr(\omega\cdot)\un)\otimes \mathrm{id}^{\otimes(n-j)}$ a direct computation shows
\begin{equation*}
	\langle x,\L_\omega(x)\rangle_\omega=\sum_{j=1}^n \langle d_j(x),d_j(x)\rangle_\omega,
\end{equation*}
which settles \eqref{(A2-1)}. Condition \eqref{(A2-2)} with $M=\sqrt{2}$ follows as in Equation \eqref{ineq:claim_bis}. Condition \ref{(A3)} with $\lambda=1$ is easy to check for $n=1$, one for arbitrary $n$ follows by tensorization. The best constant $\alpha$ satisfying \ref{(A4)} for any $n$ has been shown in \cite[Theorems 24 \& 25]{beigi2020quantum},  whereas a lower bound on $\alpha$ was found e.g. in \cite[Theorem 9]{temme2014hypercontractivity}. A direct computation shows \ref{(A5)} with $\mu=1$.

\subsection{Quantum Ornstein--Uhlenbeck semigroup}

Next, we consider the generator of the so-called quantum Ornstein--Uhlenbeck semigroup \cite{fagnola1994quantum,cipriani2000spectral,carbone2008hypercontractivity}. The latter acts on the algebra $B(\ell^2(\mathbb{N}))$ of all bounded operators on the Hilbert space $\ell^2(\mathbb{N})$ of square-summable sequences.  Denoting by $a$ and $a^*$ the annihilation and creation operators of the quantum harmonic oscillator, which are defined by their action on a given orthonormal basis $\{|k\rangle\}_{k\in\mathbb{N}}$ of $\mathcal{H}\equiv \ell^2(\mathbb{N})\simeq L^2(\mathbb{R})$ as follows:
\begin{align*}
	a^*|k\rangle=\sqrt{k+1}|k\rangle\,,\qquad\text{ and }\qquad a|k\rangle=\left\{\begin{aligned}
		&\sqrt{k}|k-1\rangle& k\ge 1\\
		&0 &\text{ else}\, ,\end{aligned}\right.
\end{align*}
the generator of the quantum Ornstein--Uhlenbeck semigroup takes the following form at least on finite rank operators:
\begin{align*}
	\mathcal{L}(x)=\frac{\mu^2}{2}\,(a^*ax-2a^*xa+xa^*a)+\frac{\lambda^2}{2}\,(aa^*x-2axa^*+xaa^*)\,,
\end{align*}
where $\mu>\lambda>0$. Denoting $\nu=\lambda^2/\mu^2$, it has a unique invariant state
\begin{align*}
	\sigma_{\mu,\nu}:=(1-\nu)\,\sum_{n\ge 0}\,\nu^n\,|n\rangle\langle n|\,.
\end{align*}
Here we will use the notion of a Schwartz operator \cite{keyl2016schwartz}: an operator $x\in B(L^2(\mathbb{R}))$ is called a Schwartz operator if for any indices $\alpha,\beta,\alpha',\beta'\in\mathbb{N}$, 
\[\sup\big\{|\langle P^\beta Q^\alpha \psi,\,x P^{\beta'}Q^{\alpha'}\varphi\rangle|:\, \psi,\varphi\in\mathfrak{S}(\mathbb{R}),\,\|\psi\|,\|\varphi\|\le 1 \big\}<\infty\,,\]
where $\mathfrak{S}(\mathbb{R})$ denotes the set of Schwartz functions over $\mathbb{R}$, $Q:(x\mapsto\psi(x))\mapsto (x\mapsto x\psi(x))$ is the so-called position operator and $P:(x\mapsto \psi(x))\mapsto (x\mapsto -i\psi'(x))$ is the momentum operator. We denote by $\mathfrak{S}(\mathcal{H})$ the algebra of Schwartz operators.

\begin{proposition}
	The semigroup generated by $\mathcal{L}$ and derivations $d_a:=[a,\cdot]$ and $d_{a*}=[a^*,\cdot]$ satisfy the conditions \ref{(A0)}-\ref{(A5)} with respect to the algebra $\mathcal{A}\equiv \mathfrak{S}(\mathcal{H})$.
\end{proposition}
\begin{proof}
	The set $\mathfrak{S}(\mathcal{H})$ of Schwartz operators is a $*$-subalgebra of $B(L^2(\mathbb{R}))$ \cite[Lemma 3.5]{keyl2016schwartz}. Moreover for any $p\ge 1$, the set $\mathfrak{S}_0(\mathcal{H})$ of finite-rank Schwartz operators is dense in the space $\mathcal{T}_p(\mathcal{H})$ of  Schatten-$p$ operators \cite[Lemma 2.5]{keyl2016schwartz}. Therefore, since finite-rank operators are $\sigma$-weakly dense in $B(\mathcal{H})$, this also holds for $\mathfrak{S}(\mathcal{H})$. In order to show that $\mathfrak{S}(\mathcal{H})$ is invariant with respect to the semigroup generated by $\mathcal{L}$, we use tools from noncommutative Fourier analysis: given a trace-class operator $x$, its characteristic function is given by 
	\begin{align*}
		\chi_x(z):= \tr(x D(z))\,,
	\end{align*}
	where $D(z):=e^{za^*-\bar{z}a}$, for $z\in\mathbb{C}$, is the so-called one-mode displacement operator. By the quantum Plancherel identity, we have that for any two trace-class operators $x,y$ \cite{holevo2011probabilistic},
	\begin{align*}
		\tr(x^* y)=\int \frac{d^2z}{\pi}\, \overline{\chi_x(z)}\,\chi_y(z)\,.
	\end{align*}
	Moreover, the quantum Ornstein--Uhlenbeck semigroup can be represented by a family of quantum channels $e^{t\mathcal{L}}$ modelling a quantum beam-splitter of transmissivity $\eta=e^{-(\mu^2-\lambda^2)t}$ and with environment state $\sigma_{\mu,\nu}$ can be shown to induce the following action on characteristic functions:
	\begin{align*}
		\chi_x(z)\longrightarrow \chi_{e^{t\mathcal{L}^*}(x)}(z)= \chi_x(\sqrt{\eta}\,z)\,\chi_{\sigma_{\mu,\nu}}(\sqrt{1-\eta}\,z)=\chi_x(\sqrt{\eta}\,z)\,e^{-(2N+1)(1-\eta)|z|^2/2}\,,
	\end{align*}
	with $N=\frac{\lambda^2}{\mu^2-\lambda^2}$. Therefore, for any $\alpha,\beta,\alpha',\beta'\in\mathbb{N}$ and normalized $\psi,\varphi\in\mathfrak{S}(\mathbb{R})$, denoting $y:=|P^{\beta'}Q^{\alpha'}\varphi\rangle\langle P^\beta Q^\alpha \psi|$, we have
	\begin{align*}
		\langle P^\beta Q^\alpha \psi,\,e^{t\mathcal{L}}(T)\,P^{\beta'}Q^{\alpha'}\varphi\rangle&=\tr(y\,e^{t\mathcal{L}}(T)) \\
		&=\tr(e^{t\mathcal{L}^*}(y)T)\\
		&=\int\,\frac{d^2z}{\pi} \,\chi_y(\sqrt{\eta}z)\,e^{-(2N+1)(1-\eta)|z|^2/2}\,\chi_T(z)\\
		&=\int\,\frac{d^2u}{\eta\pi} \,\chi_y(u)\,e^{-(2N+1)(1-\eta)|u|^2/2\eta}\,\chi_T(u/\sqrt{\eta})\\
		&<\infty\,,
	\end{align*}
	where we used that $u\mapsto e^{-(2N+1)(1-\eta)|u^2|/2\eta}\chi_T(u/\sqrt{\eta})$ is a Schwartz function, see \cite[Proposition 3.18]{keyl2016schwartz}. Finally, by \cite[Proposition 3.14]{keyl2016schwartz}, for any $x\in\mathfrak{S}(\mathcal{H})$, $\mathcal{L}(x)$ is closable with closure in $\mathfrak{S}(\mathcal{H})$. Hence, \ref{(A0)} is satisfied for the algebra $\mathcal{A}\equiv \mathfrak{S}(\mathcal{H})$ .  Property \ref{(A1)} can be easily derived from the canonical commutation relation $[a,a^*]=\mathbb{I}$ and gives $K=(\mu^2-\lambda^2)/2$ (see e.g. \cite{cm2017gradient}). Property \ref{(A2)} is satisfied for the maps $d_a:=[a,\cdot]$ and $d_{a^*}:=[a^*,\cdot]$. The Poincar\'{e} inequality \ref{(A3)} follows from the characterization of the spectrum of the generator $\mathcal{L}$ established in \cite{cipriani2000spectral}. The hypercontractivity constant in \ref{(A4)} was estimated in \cite{carbone2008hypercontractivity}. The intertwining relation of \ref{(A5)} was found in \cite{cm2017gradient}. 
\end{proof}

%

\subsection{Group von Neumann algebras}
Let $G$ be a countable discrete group with unit $e$, $L(G)$ the group von Neumann algebra on $\ell^2(G)$ generated by $\{\lambda_g,g\in G\}$ where $\lambda$ is the left regular representation of $G$. We denote by $\tau(x)=\langle x\delta_e,\delta_e\rangle$ the canonical tracial faithful state. Here and in what follows, $\delta_g$ always denotes the function on $G$ that takes value 1 at $g$ and vanishes elsewhere.
%

A function $\psi\colon G\to [0,\infty)$ is a \emph{conditionally negative definite} (cnd) length function if $\psi(e)=0$, $\psi(g^{-1})=\psi(g)$ and 
\begin{equation*}
	\sum_{g,h\in G}\overline{f(g)}f(h)\psi(g^{-1}h)\leq 0
\end{equation*}
for every $f\colon G\to \com$ with finite support such that $\sum_{g\in G} f(g)=0$.

By Schoenberg's Theorem (see for example \cite[Theorem D.11]{BO08}), to every cnd function one can associate a $\tau$-symmetric quantum Markov semigroup on $L(G)$ given by
\begin{equation*}
	P_t \lambda_g=e^{-t\psi(g)}\lambda_g.
\end{equation*}

For a countable discrete group $G$, a {\em $1$-cocycle} is a triple $(H,\pi,b)$, where $H$ is a \emph{real} Hilbert space, $\pi\colon G\to O(H)$ is an orthogonal representation, and $b\colon G\to H$ satisfies the cocycle law: $b(gh)=b(g)+\pi(g)b(h),g,h\in G.$ To any cnd function $\psi$ on a countable discrete group $G$, one can associate with a $1$-cocycle  $(H,\pi,b)$ such that $\psi(g^{-1}h)=\|b(g)-b(h)\|^2,g,h\in G$. See \cite[Appendix D]{BO08} for more information.

Fix an orthonormal basis $(e_j)_{j\in\mathcal J}$ of $H$. In case $G$ is finite, the index set $\mathcal J$ can always be taken to be finite. Let $\mathcal A$ be the linear span of the operator $\lambda_g$, $g\in G$, and let
\begin{equation*}
d_j\colon \mathcal A\to L(G),\,d_j(\lambda_g)=\langle b(g),e_j\rangle \lambda_g.
\end{equation*}
The space $\mathcal A$ is contained in the domain of the generator $\mathcal L$ of $(P_t)$ and $\L(\lambda_g)=\psi(g)\lambda_g$ for $g\in G$. Moreover, $\Gamma(\lambda_g,\lambda_h)=\langle b(g),b(h)\rangle \lambda_{g^{-1}h}$ for $g,h\in G$.

Clearly, condition \ref{(A0)} is satisfied. Condition \ref{(A1)} is satisfied with $K=0$ \cite[Example 3.14]{WZ20}. For condition \ref{(A2)} note that if $x=\sum_g f(g)\lambda_g\in\mathcal A$, then
\begin{align*}
\sum_{j\in\mathcal J}d_j(x)^\ast d_j(x)&=\sum_{j\in\mathcal J}\sum_{g,h\in G}\overline{f(g)}f(h)d_j(\lambda_g)^\ast d_j(\lambda_h)\\
&=\sum_{j\in\mathcal J}\overline{f(g)}f(h)\sum_{g,h\in G}\langle b(g),e_j\rangle\langle b(h),e_j\rangle \lambda_{g^{-1}h}\\
&=\sum_{g,h\in G}\overline{f(g)}f(h)\langle b(g),b(h)\rangle \lambda_{g^{-1}h}\\
&=\sum_{g,h\in G}\overline{f(g)}f(h)\Gamma(\lambda_g,\lambda_h)\\
&=\Gamma(x).
\end{align*}
In particular, $d_j(x)^\ast d_j(x)\leq \Gamma(x)$ for every $j\in\mathcal J$. Moreover,
\begin{equation*}
\sum_{j\in\mathcal J}\lVert i_2(d_j(x))\rVert^2=\sum_{j\in J}\tau(d_j(x)^\ast d_j(x))=\tau(\Gamma(x))=\langle i_2(x),i_2(\mathcal L(x))\rangle.
\end{equation*}
Thus condition \ref{(A2)} holds with $M=1$. Condition \ref{(A3)} holds with the spectral gap $\lambda=\inf_{g:\psi(g) >0} \psi(g)$ of $\mathcal L$. Since $d_j P_t=P_t d_j$, condition \ref{(A5)} is always satisfied with $\mu=0$.

Condition \ref{(A4)} is known to hold for certain discrete groups. For free groups, it is known that \ref{(A4)} holds with $\alpha=2$ \cite[Theorem A]{junge2015hypercontractivity}. We refer to \cite{junge2013hypercontractivity} for more examples including triangular groups, finite cyclic groups $\mathbb{Z}_N, N\ge 6$, infinite Coxeter groups etc. with $0<\alpha<\infty$.


\section{Applications}\label{sec:applications}

\subsection{Influence and circuit complexity lower bounds}

As mentioned in the introduction, Karpovsky \cite{Karpovsky76} was the first to propose the total influence, as a measure of complexity of a function $f$. This intuition was then made rigorous in \cite{linial1993constant} and \cite{boppana1997average} where  tight circuit complexity lower bounds in terms of the total influence were derived for the complexity class $\operatorname{AC}^0$ of constant depth circuits.

Similar results were recently derived in the quantum setting. For instance, \cite{bu2022complexity} show a direct link between the notion of $L^2$-influence and the complexity of quantum circuits. More precisely, they showed that for a quantum circuit $U$, that is a unitary matrix in $M_2(\mathbb{C})^{\otimes n}$ \cite[Theorem 12]{bu2022complexity}
\begin{align*}
	\frac{1}{8}\,\operatorname{CiS}^2(U)\le \operatorname{Cost}(U)\,,
\end{align*}
where the $L^2$-circuit sensitivity $\operatorname{CiS}^2(U)$ is defined as
\begin{align*}
	\operatorname{CiS}^2(U):=\max_{\|O\|_2=1}\,\Big|\Inf^2(UOU^*)-\Inf^2(O)\Big|\,,
\end{align*}
and where $\operatorname{Cost}(U)$ refers to the cost of the circuit and was introduced in a series of seminal papers by Nielsen  and coauthors \cite{nielsen2005geometric,nielsen2006optimal,nielsen2006quantum,dowling2008geometry} as a lower bound on the minimal number of one and two-qubit gates required from a given universal gate-set in order to synthesize the unitary $U$. More precisely, given traceless self-adjoint operators $h_1,\dots ,h_m$ that are supported on $2$ qubits and normalized as $\|h_i\|=1$, the circuit cost of $U$ with respect to $h_1,\dots, h_m$ is defined as
\begin{align*}
	\operatorname{Cost}(U):=\inf\int_0^1\,\sum_{j=1}^m |r_j(s)|\,ds\,,
\end{align*}
where the infimum above is taken over all continuous functions $r_j:[0,1]\to \mathbb{R}$ satisfying
\begin{align*}
	U=\mathcal{P}\operatorname{exp}\Big(-i\int_0^1\,H(s)\,ds\Big)\,,
\end{align*}
and 
\begin{align*}
	H(s)=\sum_{j=1}^m\,r_j(s)\,h_j\,,
\end{align*}
where $\mathcal{P}$ denotes the path-ordering operator. We start by providing a simple bound on the $p$-influences for $p\in [1,2]$ (for convenience we may write $\frac{\tr_{N_j}}{2^{|N_j|}}(O)$ for $\frac{\tr_{N_j}}{2^{|N_j|}}(O)\otimes \un$):

\begin{proposition}
For any $j\in\{1,\dots, n\}$, let $N_j\subset\{1,\dots, m\}$ be the minimal set of qubits such that $\frac{\tr_j}{2}(U(\frac{\tr_{N_j}}{2^{|N_j|}}(O)\otimes  \un)U^*)=U(\frac{\tr_{N_j}}{2^{|N_j|}}(O)\otimes \un)U^*$ for any $O\in M_2(\mathbb{C})^{\otimes n}$.  Then, for any self-adjoint matrix $O\in M_2(\mathbb{C})^{\otimes n}$ with $\|O\|_2\le 1$ and all $p\in [1,2]$ we have
\begin{align*}
	\Inf^p(UOU^*)\le \sum_{i=1}^n\,\|d_i O\|_2^p\,\,|\{j:i\in N_j\}|\,.
\end{align*}
In the case $p=2$ and denoting $L:=\max_i\,|\{j:\,i\in N_j\}|$, we get
\begin{align*}
	\Inf^2(UOU^*)\le L\,\Inf^2(O)\,.
\end{align*}
\end{proposition}
\begin{proof}
	For $p\in[1,2]$ and any $O\in M_2(\mathbb{C})^{\otimes n}$ with $\|O\|_2\le 1$, we have
	\begin{align*}
		\Inf^p(UOU^*)&=\sum_{j=1}^n \|UOU^*-\frac{\tr_j}{2}(UOU^*)\|_p^p\\
		&\le \sum_{j=1}^n \|UOU^*-\frac{\tr_j}{2}(UOU^*)\|_2^p \\
		&\le \sum_{j=1}^n \|UOU^*-U\frac{\tr_{N_j}}{2^{|N_j|}}(O)U^*\|_2^p\\
		&= \sum_{j=1}^n \|O-\frac{\tr_{N_j}}{2^{|N_j|}}(O)\|_2^p\\
		&\le \sum_{j=1}^n\Big(\sum_{i\in N_j}\|O-\frac{\tr_{i}}{2}(O)\|_2^2\Big)^\frac{p}{2}\\
		&\le \sum_{j=1}^n\sum_{i\in N_j}\|O-\frac{\tr_i}{2}(O)\|_2^p\\
		&\le \sum_{i=1}^n \|O-\frac{\tr_i}{2}(O)\|_2^p\, |\{j:i\in N_j\}|
	\end{align*}
where in the second inequality above we use that the partial trace $\tr_j$ is a projection onto the algebra of operators supported on $\{j\}^c$, and therefore minimized the distance to that subalgebra. The third inequality follows from the non-primitive Poincar\'{e} inequality from Equation \eqref{poincareivan}.
\end{proof}

\begin{remark}\label{remcircuits}
	The assumption $\frac{\tr_j}{2}(U(\frac{\tr_{N_j}}{2^{|N_j|}}(O)\otimes\un)U^*)=U(\frac{\tr_{N_j}}{2^{|N_j|}}(O)\otimes  \un)U^*$ can be interpreted as a lightcone condition: let's consider for simplicity $n$ a unitary circuit in brickwork architecture of the form $U=U^{\ell} U^{\ell-1}\dots U^1$, where for each $j$, $U^{2j+1}=U_{1,2}^{2j+1}\otimes \dots \otimes U^{2j+1}_{n-1,n }$ and $U^{2j}=U_{2,3}^{2j}\otimes \dots \otimes U^{2j}_{n-2,n-1 }$, where by $U^j_{r,r+1}$ we mean a unitary with non-trivial support on qubits $r$ and $r+1$. Hence, for any set $N_1=\{1,\dots, n_1\}$ and any observable $O_{N_1^c}$ supported on $N_1^c$, 
	\begin{align*}
		UO_{N_1^c} U^*= U^\ell \dots U^1 O_{N_1^c}(U^1)^*\dots (U^\ell)^*
	\end{align*}
	is supported in the set $\{n_1-\ell+1,\dots, n\}$. Hence, for $n_1=\ell+1$, the condition holds. In other words, $n_1$ 
	scales linearly with the depth $\ell$ of the circuit $U$. The above simple argument generalizes easily to higher dimensions and general local unitary circuits. 
\end{remark}

In the case when $p=1$ we can bound the total $L^1$ output influence in terms of the total $L^1$ input influence.

\begin{proposition}
	For any $j\in\{1,\dots, n\}$, let $N_j\subset\{1,\dots, m\}$ be the minimal set of qubits such that $\frac{\tr_j}{2}(U\frac{\tr_{N_j}}{2^{|N_j|}}(O)U^*)=U\frac{\tr_{N_j}}{2^{|N_j|}}(O)U^*$ for any $O\in M_2(\mathbb{C})^{\otimes n}$ and denote $L:=\max_i\,|\{j:\,i\in N_j\}|$.  Then, for any such matrix $O\in M_2(\mathbb{C})^{\otimes n}$ 
	\begin{align*}
		\Inf^1(UOU^*)\le 2 L\,\Inf^1(O)\,.
	\end{align*}
\end{proposition}
\begin{proof}
	For any $O\in M_2(\mathbb{C})^{\otimes n}$,
	\begin{align*}
		\Inf^1(UOU^*)&=\sum_{j=1}^n \|UOU^*-\frac{\tr_j}{2}(UOU^*)\|_1\\
		&\le \sum_{j=1}^n \|UOU^*-U\frac{\tr_{N_j}}{2^{|N_j|}}(O)U^*\|_1+\|U\frac{\tr_{N_j}}{2^{|N_j|}}(O)U^*-\frac{\tr_j}{2}(UOU^*)\|_1\\
			&\le 2\,\sum_{j=1}^n \|O-\frac{\tr_{N_j}}{2^{|N_j|}}(O)\|_1\\
			&\le 2\sum_{j=1}^n\sum_{i\in N_j}\|d_i O\|_1\\
		&\le 2 L\,\Inf^1(O)
	\end{align*}
	where in the second inequality above we used the definition of $N_j$ and that the partial trace $\tr_j$ is contractive in $\|\cdot\|_1$ norm. The third inequality follows from simple triangle inequality and monotonicity of the $L^1$-norm under partial traces.
\end{proof}

Finally, we find a bound on the variation of the $L^1$-influence through a circuit $U$ in terms of the cost of a unitary $U$. We define the $L^1$-circuit sensitivity of a unitary $U\in M_2(\mathbb{C})^{\otimes n}$ as 
\begin{align*}
	\operatorname{CiS}^1(U):=\max_{\|O\|_1=1}\,\Big|\Inf^1(UOU^*)-\Inf^1(O)\Big|\,.
\end{align*}
\begin{theorem}
	The $L^1$-circuit sensitivity of a unitary $U\in M_2(\mathbb{C})^{\otimes n}$ is a lower bound on the circuit cost:
\begin{align*}
	\operatorname{CiS}^1(U)\le 8\,\operatorname{Cost}(U)\,.
\end{align*}
\end{theorem}

\begin{proof}
	Our proof follows similar steps to those leading to \cite[Theorem 12]{bu2022complexity} (see also \cite{eisert2021entangling,marien2016entanglement}): we first show that, for a unitary $U_t=e^{-itH}$, where $H$ acts non-trivially on a set $S$ of $k$ qubits, and $O$ with $\|O\|_1=1$, 
	\begin{align}
		|\Inf^1(U_tOU_t^*)-\Inf^1(O)|
		&=|\sum_{j=1}^n \|d_j(O(t))\|_1-\|d_j(O)\|_1|\nonumber\\
		&=|\sum_{j\in S}\,\|d_j(O(t))\|_1-\|d_j(O)\|_1|\nonumber\\
	&\le \sum_{j\in S}\,\|d_j(O(t))-d_j(O)\|_1\nonumber\\
		&\le \int_0^t\,\sum_{j\in S}\,\|d_j\,e^{-isH}[H,O]e^{isH}\|_1\,ds\nonumber\\
		&\le 2kt\, \|[H,O]\|_1\nonumber\\
		&\le 4kt \lVert H\rVert\lVert O\rVert_1\nonumber\\
		&= 4kt\,\|H\|\,,\label{localH}
	\end{align}
	where we denoted $O(t):=U_tOU_t^*$. Back to our original problem, we take a Trotter decomposition of $U$ such that for arbitrary small $\epsilon>0$, 
	\begin{align*}
		\|U-V_N\|\le \epsilon
	\end{align*}
where $V_N$ is defined as follows:
\begin{align*}
	&V_N:=\prod_{j=1}^N\,W_j\,,\\
	&W_\eta:=\operatorname{exp}\Big( -\frac{i}{N}\sum_{j=1}^m\,r_j\big(\frac{\eta}{N}\big)\,h_j\Big)\,,
\end{align*}	
so that
\begin{align*}
	&W_\eta=\lim_{l\to\infty}W_\eta^{(l)}\,,\\
	&W_\eta^{(l)}:=\Big(W_{\eta,1}^{\frac{1}{l}}\dots W_{\eta,l}^{\frac{1}{l}}\Big)^l\,,\\
	&W_{\eta,j}:=\operatorname{exp}\Big(-\frac{i}{N}\,r_j\big(\frac{\eta}{N}\big)\,h_j\Big)\,.
\end{align*}
Next, we define $O_\eta=W_\eta O_{\eta -1}W_\eta^*$ with $O_0=O$. We have,
\begin{align*}
	|\Inf^1(O_\eta)-\Inf^1(O_{\eta-1})|
	&=|\Inf^1(W_\eta O_{\eta-1}W_\eta^*)-\Inf^1(O_{\eta-1})|\\
	&=\limsup_{l\to\infty}\, |\Inf^1\big(W_\eta^{(l)}O_{\eta-1}W_\eta^{(l)\ast}\big)-\Inf^1(O_{\eta-1})|\\
	&\le \limsup_{l\to\infty}\frac{l}{N}\,\sum_{j=1}^m\,\frac{8}{l}\,\Big|r_j\Big(\frac{\eta}{N}\Big)\Big|\\
	&=\frac{8}{N}\,\sum_{j=1}^m\,\Big|r_j\Big(\frac{\eta}{N}\Big)\Big|\,,
\end{align*}
where the inequality follows from \eqref{localH} for $k=2$ and $t=\frac{1}{N}$. Summing over $\eta$, we get
\begin{align*}
	|\Inf^1(UOU^*)-\Inf^1(O)|
	\le \frac{8}{N}\,\sum_{\eta=1}^N\sum_{j=1}^m\,\Big|r_j\Big(\frac{\eta}{N}\Big)\Big|+	|\Inf^1(UOU^*)-	\Inf^1(V_NOV_N^*)|.
\end{align*}
Since the circuit cost is expressed as $\operatorname{Cost}(U)=\inf_{(r_j)_j} \lim_{N\to\infty}\frac{1}{N}\sum_{\eta=1}^N\sum_{j=1}^m\Big|r_j\Big(\frac{\eta}{N}\Big)\Big|$, and since the influence of $UOU^*$ can be arbirarily well approximated by that of $V_NOV_N^*$ as $N\to\infty$, the result follows.
\end{proof}

\begin{remark}
	Combined with our quantum Friedgut's Junta theorem \ref{thm:L1friedgut}  the above results show that for any observable $O$ with $\Inf^1(O),\Inf^2(O)=\mathcal{O}(1)$ and $\|O\|_2=\mathcal{O}(1)$, and for any unitary $U$ with $L=\mathcal{O}(1)$, the output observable $UOU^*$ can be well approximated by a $k$-junta with $k=\mathcal{O}(1)$. Taking again the simple example constructed in Remark \ref{remcircuits}, we recover the simple fact that, for a $1$-qubit Pauli matrix evolving according to a circuit of constant depth, the output observable will still be supported on a constant size region. While it would be interesting to find some non-trivial situations where our bounds still hold, we leave this question to future work. 
\end{remark}

\subsection{Learning quantum Boolean functions}\label{sec:learning}

In this section, we use our quantum Friedgut's Junta theorem \ref{thm:L1friedgut} to provide an efficient algorithm for learning quantum Boolean functions. Our argument relies on the following quantum generalization of Goldreich--Levin theorem (see  \cite[Theorem 7.6]{mo10quantumboolean}):
\begin{theorem}[quantum Goldreich--Levin]\label{quantumGoldreichLevin} Given an oracle access to a unitary operator $U$ on $n$ qubits and its adjoint $U^*$, and given $\delta,\gamma>0$, there is a $\operatorname{poly}\big(n,\frac{1}{\gamma}\big)\log\big(\frac{1}{\delta}\big)$-time algorithm which outputs a list $L=\{s_1,\dots, s_m\}$ such that with probability $1-\delta$: (1) if $|\widehat{U}_s|\ge \gamma$ then $s\in L$; and (2) for all $s\in L$, $|\widehat{U}_s|\ge \gamma/2$.
	
\end{theorem}

Once the quantum Goldreich--Levin algorithm has been used to output a list of Fourier coefficients, the following lemma, which is also taken from \cite{mo10quantumboolean}, can be used to compute them:

\begin{lemma}\label{lemmacomputefourier}\cite[Lemma 7.4]{mo10quantumboolean}
For any quantum Boolean function $A$, and any $s\in \{0,1,2,3\}^n$ it is possible to estimate $\widehat{A}_s$ to within $\pm \eta$ with probability $1-\delta$ using $\mathcal{O}\big(\frac{1}{\eta^2}\log\big(\frac{1}{\delta}\big)\big)$ queries.
\end{lemma}

Combining Theorem \ref{quantumGoldreichLevin}, Lemma \ref{lemmacomputefourier} with our  Theorem \ref{thm:L1friedgut}, we directly arrive at the following result:

\begin{proposition}[Learning quantum Boolean functions]\label{proplearning}
	Let $A\in M_2(\com)^{\otimes n}$ be a quantum Boolean function. Given oracle access to $A$, with probability $1-\delta$, we can learn $A$ to precision $\epsilon$ in $L^2$ using $\operatorname{poly}(n,4^k,\log\big(\frac{1}{\delta}\big))$ queries to $A$, where 
	\begin{align*}
	k\le k(\epsilon)\equiv \left\{ \begin{aligned}
	&\Inf^1(A)^2e^{\frac{48\Inf^2(A)}{\epsilon^2}\log\frac{2\Inf^2(A)}{\epsilon}}& \text{ if }\,\Inf^2(A)\ge 1\,;	\\
	&\frac{\Inf^1(A)^2}{\Inf^2(A)}\,e^{\frac{48\Inf^2(A)}{\epsilon^2}\log\frac{2\sqrt{\Inf^2(A)}}{\epsilon}}& \text{ else}
\end{aligned}\,\right.
	\end{align*}

\end{proposition}

\begin{proof}
	By Theorem \ref{thm:L1friedgut}, we have a quantum Boolean function $B=2^{-|T|}\tr_{T}(A)\in M_2(\com)^{\otimes n}$ supported on a region $T^c$ of $k$ qubits  such that $\|A-B\|_2\le \epsilon$, with $k\le k(\epsilon)$. Next, we use the quantum Goldreich--Levin algorithm of Theorem \ref{quantumGoldreichLevin} in order to output a list $L=\{s^{(1)},\dots , s^{(m)}\}\subseteq\{0,1,2,3\}^n$ of corresponding significant Fourier coefficients $\widehat{A}_{s^{(j)}}$, $1\le j\le m$. Denoting the operator $A_L:=\sum_{s\in L} \widehat{A}_{s}\sigma_{s}$ as well as the set of strings $s\in \{0,1,2,3\}^n$ with $s_T=0$ as $L'$, we get

\begin{align*}
	\|A-A_L\|_2&= \|\sum_{s\in L^c}\widehat{A}_s \sigma_s\|_2\\
&=\Big(\sum_{s\in L^c}\widehat{A}_s^2\Big)^{1/2}\\
&=\Big(  \sum_{s\in (L\cup L')^c} \widehat{A}_s^2+\sum_{s\in L'\setminus L}\widehat{A}_s^2\Big)^{\frac{1}{2}}\\
&\le \Big(  \sum_{s\in  L'^c} \widehat{A}_s^2+\sum_{s\in L'\setminus L}\widehat{A}_s^2\Big)^{\frac{1}{2}}\\
&\le \Big( \epsilon^2+\sum_{s\in L'\setminus L}\widehat{A}_s^2\Big)^{\frac{1}{2}}\\
\end{align*}
Moreover, by Theorem \ref{quantumGoldreichLevin}, we have that with probability $1-\delta$, if $s\notin L$, then $|\widehat{A}_s|\le \gamma$. Therefore we have that with that same high probability
\begin{align}
	\|A-A_L\|_2\le \big(\epsilon^2+4^k\, \gamma^2  \big)^{\frac{1}{2}}\,.\label{eq:boundAAL}
\end{align}
It remains to evaluate the coefficients $\widehat{A}_s$ for $s\in L$. This can be done within precision $\pm\eta$ with probability $(1-\delta)$ using $\mathcal{O}\big(\frac{1}{\eta^2}\log\big(\frac{1}{\delta}\big)\big)$ queries of $A$ according to Lemma \ref{lemmacomputefourier}. Moreover, since $A$ is a quantum Boolean function, there are at most $\frac{4}{\gamma^2}$ coefficients $\widehat{A}_s$ such that $|\widehat{A}_s|\ge \frac{\gamma}{2}$. Therefore, with high probability $|L|\le \frac{4}{\gamma^2}$. Choosing $\gamma=\epsilon 2^{-k}$ so that $\|A-A_L\|_2= \mathcal{O}( \epsilon)$, we need to evaluate $4/\gamma^2=\mathcal{O}(4^k)$ such coefficients. The result follows.
\end{proof}

\subsection{Learning quantum dynamics}

Proposition \ref{proplearning} extends the domain of applicability of Proposition 41 in \cite{MO10arxiv} where the authors provided an efficient algorithm to learn the evolution of initially local observables under the dynamics generated by a local Hamiltonian. While the proof of \cite[Proposition 41]{MO10arxiv} requires the Lieb--Robinson bounds in order to control the sets of sites of large influence in terms of the support of the initial observable and the light-cone of $H$, our argument has the advantage of not putting any geometric locality assumption of the quantum Boolean function $A$.

To further illustrate our result, we consider the following generalization of the setup of \cite[Proposition 41]{mo10quantumboolean}: let $\Lambda$ be a finite set of size $|\Lambda|=n$ endowed with a metric $d(\cdot,\cdot)$. We suppose that there is a monotone increasing function $g$ on $[0,\infty)$ and constants $C,D>0$ such that 
\begin{align*}
	\big|\big\{y\in\Lambda\big|\,d(x,y)\le r\big\}\big|\le g(r)\le C(1+r)^D, \quad r\ge 0, x\in\Lambda\, .
\end{align*}
The constant $D$ typically denotes the spatial dimension in the case of a regular lattice. We consider a quantum spin system on the point set $\Lambda$ by assigning the Hilbert space $\mathcal{H}_x\equiv \mathbb{C}^2$ to each site $x\in\Lambda$. For any subset $T \subseteq \Lambda$, the configuration space of spin states on $T$ is given by the tensor product $\mathcal{H}_T=\bigotimes_{x\in T}\mathcal{H}_x$, and the algebra $\mathcal{A}_T:=B(\mathcal{H}_T)$ of observables on T acts on the Hilbert space $\mathcal{H}_T$. We consider a Hamiltonian $H_\Lambda=\sum_{X\subseteq \Lambda} h_X$, where $h_X\in B(\mathcal{H}_X)$ is a local Hamiltonian, i.e. a self-adjoint operator on $\mathcal{H}_X$, for each $X\subset \Lambda$. In what follows, we denote the diameter of a set $Z\subseteq \Lambda$ by $\operatorname{diam}(Z):=\max\{d(x,y)|\,x,y\in Z\}$. We further assume the following requirements \cite[Assumption A]{matsuta2017improving}:
\begin{itemize}
	\item[(i)] There is a decreasing function $f(R)$ on $[0,\infty)$, such that 
	\begin{align*}
		\max_{x\in\Lambda}\,\sum_{\substack{Z\ni x\\\operatorname{diam}(Z)\ge R}}\|h_Z\|\le f(R)\,,\quad R\ge 0;
	\end{align*} 
\item[(ii)] The following constant is independent of the system size $n$:
\begin{align*}
	\mathcal{C}_0:=\max_{x\in\Lambda}\,\sum_{y\in\Lambda}\sum_{Z\ni x,y}\|h_Z\|<\infty\,.
\end{align*} 
\end{itemize} 
Strictly speaking, condition (ii) only makes sense when considering a family of Hamiltonians $H_\Lambda$ defined on an increasing family of sets $\Lambda$ all included in a countable set $\Sigma$. We will however favour simplicity over rigour here.

By \cite[Theorem 2.1]{matsuta2017improving}, for any two one-local Pauli operators $\sigma_{s_i},\sigma_{s_j}$ with $j\ne i$, and all $R\ge 1$ we have that, given $d_{ij}\equiv d(i,j)$
\begin{align*}
	&\big\|\big[e^{itH_\Lambda}\sigma_{s_i}e^{-itH_\Lambda},\,\sigma_{s_j}\big] \big\|\\
	&\quad \qquad\qquad\le 2 e^{vt-d_{ij}/R}+4t\,g(d_{ij})\,f(R)+2\mathcal{C}_2\,t\,R\,\max\{d_{ij},R\}^D\,f(R)\,e^{vt-d_{ij}/R}
\end{align*}
for any $t\ge 0$, where $v$ and $\mathcal{C}_2$ are positive constants independent of $\Lambda,t,R,i$ and $j$.

\begin{proposition}\label{prop:dynamics}
	With the above assumptions, we further assume that there is $R\ge 1$ such that for all $i\in \Lambda$, the constants 
	\begin{align*}
	C_i:= \sum_{j\in\Lambda}\, e^{vt-d_{ij}/R}+ t\,g(d_{ij})\,f(R)+t\,R\,\max\{d_{ij},R\}^D\,f(R)\,e^{vt-d_{ij}/R}
\end{align*}	
can be bounded by constants independent of the size $n$ of the system. Then, with probability $1-\delta$, we can learn the quantum Boolean functions $e^{itH_\Lambda}\sigma_{s_i}e^{-itH_\Lambda}$ to precision $\epsilon$ in $L^2$ using $\operatorname{poly}(n,\exp(\exp(\epsilon^{-2}|\log(\epsilon)|)),\log\big(\frac{1}{\delta}\big))$ queries to $e^{-itH_\Lambda}$ and $e^{itH_\Lambda}$.
	\end{proposition}

\begin{proof}
	In view of the dependence of $k$ on the influences in Proposition \ref{proplearning} on the influences, it is enough to control	$\Inf^1(e^{itH_\Lambda}\sigma_{s_i}e^{-itH_\Lambda})$ and $\Inf^2(e^{itH_\Lambda}\sigma_{s_i}e^{-itH_\Lambda})$  independently of the size of the system. We clearly have for any $A\in M_2(\com)^{\otimes n}$ that $\Inf^1(A)\le \sum_{j\in\Lambda} \|d_j A\|$ and $\Inf^2(A)\le \sum_{j\in\Lambda} \|d_j A\|^2$. Moreover, by the following well-known expression for the partial trace 
\begin{align*}
	\frac{1}{2}\tr_j(A)\otimes  \un_j=\frac{1}{4}\,\sum_{s_j}\,\sigma_{s_j}A\sigma_{s_j}\,,
\end{align*}
where $\sigma_{s_j}$ are Pauli matrices on site $j$, we have that  
\begin{align*}
	\|d_j A\|= \Big\| A-\frac{1}{4}\sum_{s_j}\sigma_{s_j}A\sigma_{s_j}\Big\|&=\frac{1}{4}\,\Big\|3A-\sum_{s_j:\sigma_{s_j}\neq \un}\sigma_{s_j}A\sigma_{s_j}\Big\|\le \frac{1}{4}\,\sum_{s_j:\sigma_{s_j}\neq \un}\,\|[A,\,\sigma_{s_j}]\|\,.
\end{align*}
 Therefore, 
\begin{align*}
	&\sqrt{\Inf^2\big(e^{itH_\Lambda}\sigma_{s^{(i)}}e^{-itH_\Lambda}\big )},\,\Inf^1\big(e^{itH_\Lambda}\sigma_{s^{(i)}}e^{-itH_\Lambda}\big )\\
	&\quad \quad\qquad\qquad \le \frac{3}{2} \sum_{j\in\Lambda}\, e^{vt-d_{ij}/R}+2t\,g(d_{ij})\,f(R)+\mathcal{C}_2\,t\,R\,\max\{d_{ij},R\}^D\,f(R)\,e^{vt-d_{ij}/R}\,.
\end{align*}
\end{proof}

\begin{remark}
	Our result in Proposition \ref{proplearning} has the advantage that it does not assume in advance that $A$ is (close to) a $k$-junta. This comes at the price that the dependence of the query complexity on the approximating parameter $\epsilon$ scales doubly exponentially with the latter. For the same reason, our dynamics learning method in Proposition \ref{prop:dynamics} allows us to extend the class of Hamiltonians considered in \cite{mo10quantumboolean} to Hamiltonians satisfying a weaker power-law decay, at the cost of a much worse dependence on $\epsilon$. This dependency on $\epsilon$ is also not new in classical setting \cite{bshouty1996fourier,o2007learning}.
\end{remark}

\begin{remark}
	In a recent article \cite{chen2022testing}, the authors provide an algorithm for learning any unitary $k$-junta $U$ with precision $\epsilon$ and high probability which uses $\mathcal{O}\big(\frac{k}{\epsilon}+\frac{4^k}{\epsilon^2}\big)$ queries to $U$
	(see Theorem 29), extending a previous quantum algorithm for learning classical $k$-juntas reported in \cite{atici2007quantum}. While the dependence on $\epsilon$ is much tighter than ours, the two results are incomparable, since we replaced the requirement that $U$ is a $k$-junta by the weaker condition that it has influences $\Inf^1 U, \Inf^2 U=\mathcal{O}(1)$. 
\end{remark}

\subsection{Quantum isoperimetric type inequalities}\label{subsect:isoperimetry}

Closely related to the concentration of measure phenomenon and functional inequalities, isoperimetric inequalities provide powerful tools in the analysis of extremal sets and surface measures. Given a metric space $(X,d)$ equipped with a Borel measure $\mu$, the \textit{boundary measure} of a Borel set $A$ in $X$ with respect to $\mu$ is defined as \cite{ledoux2000geometry,ledoux2001concentration,BGL2014book}
\begin{align*}
	\mu^+(A)=\lim_{r\to 0} \frac{1}{r}\,\mu(A_r\backslash A)
\end{align*}
where we recall that $A_r:=\{x\in X|\,d(x,A)<r\}$ is the (open) $r$-neighbourhood of $A$. The \textit{isoperimetric profile} of $\mu$ corresponds to the largest function $I_\mu$ on $[0,\mu(X)]$ such that, for any Borel set $A\subset X$ with $\mu(A)<\infty$,
\begin{align}\label{def:isop}
\mu^+(A)\ge I_\mu(\mu(A))	\,.
	\end{align}

In the case of the canonical Gaussian measure $\gamma$ on the Borel sets of $\mathbb{R}^k$ with density $(2\pi)^{-k/2}e^{-|x|^2/2}$ with respect to the Lebesgue measure, with the usual Euclidean metric induced by the norm $|x|$ \cite[Theorem 2.5]{ledoux2001concentration}:
\begin{align*}
	I_\gamma= \Phi'\circ \Phi^{-1}\,
\end{align*}
where $\Phi(t)=(2\pi)^{-1/2}\int_{-\infty}^te^{-x^2/2}dx$ is the distribution function of the canonical Gaussian measure in dimension one. Moreover, equality holds in (\ref{def:isop}) if and only if $A$ is a half-space in $\mathbb{R}^k$. Moreover, as $a\to 0$, we have 
\begin{align}\label{gauss}
 \Phi'\circ \Phi^{-1}(a)\sim a\Big(2\log\frac{1}{a}\Big)^{\frac{1}{2}}\,.
\end{align}
  Similar isoperimetric inequalities were also derived for hypercontractive, log-concave measures \cite{bakry1996levy} (see also \cite{milman2009role,milman2010isoperimetric}). When $k=1$, the boundary measure of a Borel set $A$ can be expressed in terms of the geometric influence $\| f_A'\|_1$ of a smooth approximation $f_A$ of the characteristic function of $A$. In other words, $\mu^+(A)\approx \Inf^1(f_A)$.

This observation allows us to generalize the notion of isoperimetric inequality in the context of smooth Riemannian manifolds to discrete settings. In the context of the classical Boolean hypercube $\Omega_n$, the \textit{edge isoperimetric inequality} states that for any $m$, among the $m$-element subsets of the discrete cube, the minimal edge boundary is attained by the set of $m$ largest elements in the lexicographic order
\cite{bernstein1967maximally,harper1964optimal,hart1976note,lindsey1964assignment}. In particular, for any set $A\subset \Omega_n$ of vertices
\begin{align}\label{hypercubeisop}
\mu_n(\partial A)\equiv \Inf(f_A) \ge 	2\mu_n(A)\log_2\left(\frac{1}{\mu_n(A)}\right)\,,
\end{align}
where we recall that $\mu_n$ is the uniform probability measure on $\Omega_n$, and $f_A$ corresponds to the characteristic function of set $A$. Here, $\partial A$ simply corresponds to the set of  vertices in the complement of $A$ that are adjacent to $A$. This inequality is moreover tight when $|A|=2^d$ for some $d\in\mathbb{N}$ (take for instance $A$ to be the vertices of a $d$-dimensional subcube). We notice the similarity with (\ref{gauss}) up to the change of power in the logarithmic factor. 

Similarly, consider a finite graph $G=(V,E)$ with set of vertices $V$ and set of edges $E$ with bounded degree $d$ (i.e. each vertex has at most a fixed number $d$ adjacent edges). The graph $G$ is said to satisfy the \textit{linear isoperimetric inequality} if
\begin{align*}
	\operatorname{Card}(\partial A)\ge h\,\operatorname{Card}(A)\,,
\end{align*}
for some $h>0$ and all subsets $A$ of $V$ such that $\operatorname{Card}(A)\le \frac{1}{2}\operatorname{Card}(V)$. The so-called \textit{Cheeger} constant $h$ of the graph can be related to the spectral gap $\lambda$ of the graph Laplacian via Cheeger's and Buser's inequalities. Here,  $\operatorname{Card}(\partial A)$ plays the role of $\mu^+(A)$ and can be once again related to a notion of influence. The linear isoperimetric inequality can be understood as a weaker form of isoperimetry than the one derived for log-concave, hypercontractive measures, and hence only implies exponential concentration for the normalized counting measure on $G$. Moreover, one should not expect to recover the stronger Gaussian type isoperimetry in this setting, since the hypercontractivity constant for graph Laplacians is known to scale with the size of the graph \cite{bobkov2006modified}.

Linear isoperimetric inequalities were also considered in the more general context of Markov chains over finite sample spaces. For instance, in the case of a continuous time Markov chain with transition rates $Q(x,y)$ and unique reversible probability measure $\pi$ with non-negative entropic Ricci curvature, \cite{erbar2018poincare} established that for any set $A$,
\begin{align*}
	\pi^+(\partial A)\ge \frac{1}{3}\,\sqrt{Q_*\lambda}\,\,\pi(A)(1-\pi(A))\,
\end{align*}
where $\lambda$ is the spectral gap of $Q$, $Q_*=\min\{Q(x,y):Q(x,y)>0\}$ and $\pi^+(\partial A)=\sum_{x\in A,y\in A^c}Q(x,y)\pi(x)$ denotes the perimeter measure of $A$. We also note that extensions of such inequalities in the quantum setting were obtained in \cite{temme2010chi}. 

 Interestingly, such inequalities are well-known to be equivalent to an $L^1$-Poincar\'{e} inequality. It is then natural to ask whether one could recover the type of isoperimetry found for the Gaussian measure and uniform measure on the hypercube in other discrete and quantum settings by further assuming hypercontractivity of the (quantum) Markov chain. 
   This is indeed the case, as we prove by a direct appeal to Talagrand's inequality:
 
 \begin{theorem}[Qubit isoperimetric type inequality] For any projection $P_A$ onto a subspace $A\subset  (\mathbb{C}^2)^{\otimes n}$, 
 	\begin{align}\label{eq:L1isop}
\max_{1\le j\le n}\Inf^1_j(P_A)\ge \frac{C}{n}\,\tau(A)(1-\tau(A)) \log\left(\frac{n}{\tau(A)(1-\tau(A)) }\right)^{\frac{1}{2}}\,.
 	\end{align}
 	for some universal constant $C$, where $\tau(A):=2^{-n}\tr (P_A)$.
 
 \end{theorem}

\begin{proof}
	As mentioned in \cite{cel12talagrand}, this is a simple corollary of Talagrand's inequality Theorem \ref{quantumTalagrand} after assuming that 
	\begin{align*}
		\Inf_j^1(P_A)\le \Big(\frac{\tau(A)(1-\tau(A))}{n}\Big)^{\frac{1}{2}}
		\end{align*}
		for every $j\in\{1,\dots,n\}$, since otherwise the result directly holds.
\end{proof}

\begin{remark}
	 Similar to the quantum KKL conjecture of Montanaro and Osborne, it is reasonable to conjecture the following $L^2$ variant of (\ref{eq:L1isop})
	\begin{align}\label{eq:L2isop}
	\max_{1\le j\le n} \Inf^2_j(P_A)\ge \frac{C}{n}\,\tau(A)(1-\tau(A)) \log\left(\frac{n}{\tau(A)(1-\tau(A)) }\right)\,.
\end{align}
\end{remark}

We end this section by remarking the following $L^1$-Poincar\'e inequality that is stronger than Theorem \ref{thm:poincar}. See \cite{ivanisvili2018improving} for the discussions on the classical Boolean cubes.

\begin{theorem}\label{thm:strong_L1_poincare}
	For all $A\in M_2(\com)^{\otimes n}$, one has 
	\begin{equation}\label{ineq:strong_L1_poincare}
	\|A-2^{-n}\tr(A)\|_1\le \pi \left\|\left(\sum_{j=1}^{n}d_j(A)^\ast d_j(A)\right)^{1/2}\right\|_1\le \sqrt{2}\pi \|\Gamma(A)^{1/2}\|_1\, .
	\end{equation}
\end{theorem} 

\begin{remark}
	Using the inequality 
	$$\|X+Y\|_p^p\le \|X\|_p^p+\|Y\|_p^p,\qquad 0<p<1$$
	 for $p=1/2$, we get
	\begin{equation*}
	\left\|\left(\sum_{j=1}^{n}d_j(A)^\ast d_j(A)\right)^{1/2}\right\|_1
	\le \sum_{j=1}^{n}\left\|\left(d_j(A)^\ast d_j(A)\right)^{1/2}\right\|_{1}
	=\sum_{j=1}^{n}\left\| d_j(A)\right\|_{1}\, .
	\end{equation*}
	So \eqref{ineq:strong_L1_poincare} is stronger than \eqref{ineq:l1poincare} up to the multiplicative constant $\pi$.
\end{remark}

\begin{proof}[Proof of Theorem \ref{thm:strong_L1_poincare}]
	Recall that we have proved \eqref{ineq:claim_bis}
	\begin{equation*}
	\sum_{j=1}^{n}d_j(A)^\ast d_j(A)\le 2\Gamma(A) \, .
	\end{equation*}
	So we have the second inequality of \eqref{ineq:strong_L1_poincare} by operator monotonicity of $x\mapsto x^{1/2}$. Now let us prove the first inequality of \eqref{ineq:strong_L1_poincare}. We denote by $(A,B)\mapsto \langle A,B\rangle:=2^{-n}\tr(A^* B)$ the normalized Hilbert--Schmidt inner product.  
	By semigroup interpolation \eqref{eq:defn of semigroup} and duality,
	\begin{align*}
		\|A-2^{-n}\tr(A)\|_1
		&=\sup_{\|B\|\le 1}|\langle A-2^{-n}\tr(A),B\rangle| \\
		&\overset{\eqref{eq:defn of semigroup}}{=}\sup_{\|B\|\le 1}\left|\int_0^\infty\langle \mathcal L P_t(A),B\rangle dt\right| \\
		&\le \sup_{\|B\|\le 1}\int_0^\infty\sum_{j=1}^{n}\left|\langle d_j A, d_j P_t B\rangle \right|\,dt\,.
	\end{align*}
	Now recall that we have the following inequality: 
	\begin{equation}\label{ineq:vector_valued}
	\left|\sum_{j=1}^{n}\langle X_j, Y_j\rangle\right|\le \left\|\left(\sum_{j=1}^{n} X_j^\ast X_j\right)^{1/2}\right\|_1\cdot \left\|\sum_{j=1}^{n}Y^\ast_j Y_j\right\|^{1/2} \,.
	\end{equation}
	To prove it, form the operators $X:=\sum_{j=1}^{n}X_j\otimes \ket{j}\bra{1}$ and $Y:=\sum_{j=1}^{n}Y_j\otimes \ket{j}\bra{1}$. Then H\"older's inequality gives 
	\begin{equation*}
	\|X^\ast Y\|_1\le \|X\|_1\cdot \|Y\|=\|\left(X^\ast X\right)^{1/2}\|_1\cdot \|Y^\ast Y\|^{1/2}\, ,
	\end{equation*}
	which, together with
	$$\left|\sum_{j=1}^{n}\langle X_j, Y_j\rangle\right|\le \left\|\sum_{j=1}^{n}X_j^\ast Y_j\right\|_1=\|X^\ast Y\|_1\, ,$$
	yields \eqref{ineq:vector_valued}.
	Now apply \eqref{ineq:vector_valued} to $(X_j,Y_j)=(d_j A, d_j P_t B)$ to get
	\begin{equation*}
	\left|\sum_{j=1}^{n}\langle d_j A, d_j P_t B\rangle \right|
	\le \left\|\left(\sum_{j=1}^{n}d_j(A)^\ast d_j(A)\right)^{1/2}\right\|_1\cdot \left\|\sum_{j=1}^{n}d_j (P_t B)^\ast d_j (P_t B) \right\|^{1/2}\, .
	\end{equation*}
To conclude, we use Lemma \ref{lem:technical lemma} together with
\begin{equation*}
	   \int_0^\infty \frac{dt}{\sqrt{e^t-1}}	=\pi\,. 
	\qedhere
	\end{equation*}

\end{proof}

\section{Discussions}

	We end this paper with the following discussions. 
	 
	 \subsection{Equivalence between log-Sobolev and Talagrand's inequalities }
In Theorem \ref{thm:general_Talagrand}, we derived a general noncommutative extension of Talagrand's inequality. Our proof requires the joint use of the hypercontractivity inequality \ref{(A4)} with the intertwining relation \ref{(A5)}. It is hence legitimate to ask whether, in return, such Talagrand-type inequalities imply hypercontractivity. This question was answered in the positive in the classical, continuous setting in \cite[Proposition 1]{BOBKOV1999281}, and later on for discrete spaces in \cite{vollering2016talagrand}. 
It would be interesting to consider the similar problem in the quantum setting, which we leave to  future work.

	 \subsection{Learning low-degree  quantum Boolean functions}
	 	An alternative notion of complexity than the support condition for $k$-juntas is that of the degree: a bounded function $f:\Omega_n\to [-1,1]$ is said to have degree at most $d\in \{1,\dots, n\}$ if for any string $s\in \{-1,1\}^n$ with Hamming weight $|s|>d$, the Fourier coefficient $\widehat{f}(s)=0$. In particular, Boolean functions of degree at most $d$ are $d2^{d-1}$ juntas \cite{nisan1994degree}. As a main tool for the result, the authors derived a simple lower bound on the degree of the function in terms of its total influence. This observation can be used in conjunction with the Goldreich--Levin algorithm in order to devise a learning algorithm which makes $\operatorname{poly}(n)$ random queries to $f$. More efficient algorithms were proposed in the past decades \cite{linial1993constant,iyer2021tight,mansour1994learning}. However all these algorithm have a query complexity scaling polynomially with $n$.  In the recent article \cite{eskenazis2022learning}, the authors show that any low degree Boolean function can be approximated to $\epsilon$ precision in $L^2$ with probability $1-\delta$ from $\mathcal{O}\big(\operatorname{poly}\big(\frac{1}{\epsilon},d\big)\log\big(\frac{n}{\delta}\big)\big)$ random queries to the function. While this result is incomparable to the ones we report in Section \ref{sec:learning}, it would be interesting to find a quantum extension of it.
The result of \cite{eskenazis2022learning} uses the so-called {\em Bohnenblust--Hille inequalities}. The study of  Bohnenblust--Hille inequalities has a long history and these inequalities have found many applications in various problems. A Boolean analogue was known \cite{defant2019fourier} and has led to interesting applications to learning theory \cite{eskenazis2022learning}. Here we formulate and conjecture a quantum analogue of Bohnenblust--Hille inequality and explain why it is useful to learning problems in the quantum setting.
	 
	 \begin{conjecture}\label{conjectureqBH}
	 	Fix $d\ge 1$. Then there exists $C_d>0$ depending only on $d$ such that for all $n\ge 1$ and all $A\in M_2(\com)^{\otimes n}$ of degree at most $d$ i.e. 
	 	$$A=\sum_{s\in\{0,1,2,3\}^n:|s|\le d}\widehat{A}_s\sigma_s,$$ 
	  we have
	 	\begin{equation}\label{eqconjecture}
	 	\left(\sum_{s\in\{0,1,2,3\}^n:|s|\le d}|\widehat{A}_s|^{\frac{2d}{d+1}}\right)^{\frac{d+1}{2d}}\le C_d \|A\|\, .
	 	\end{equation}
where the degree $|s|$ of a string $s$ is defined as the number of components that are different from $0$.
	 \end{conjecture}

If Conjecture \ref{conjectureqBH} holds, we expect that it can be used in a similar fashion as in \cite{eskenazis2022learning} in order to devise a highly efficient algorithm for learning quantum Boolean functions of small degree in terms of query complexity.

	In fact, this conjecture has been resolved after an earlier version of this paper was post out. It was first resolved by Huang, Chen and Preskill \cite{HCP22}. Later on, another proof was found by Volberg and Zhang \cite{VZ22}.

\subsection*{Data availability statement}
No data was generated as part of this work. 

\appendix
\section{Proofs from Section \ref{sec:generalization}}

\begin{proof}[Proof of Lemma \ref{lem:reverse_Poincare}]
	Let $F(s)=P_s(P_{t-s}(x)^\ast P_{t-s}(x))$. Since $x\in\mathcal A$, the map $F$ is $\sigma$-weakly differentiable on $[0,t]$ and
	\begin{align*}
		F'(s)&=-P_s\L(P_{t-s}(x)^\ast P_{t-s}(x))+P_s((\L P_{t-s}(x))^\ast P_{t-s}(x))+P_s(P_{t-s}(x)^\ast \L P_{t-s}(x))\\
		&=2P_s\Gamma(P_{t-s}(x)).
	\end{align*}
	Hence if $\omega\in \mathcal M_\ast$ is positive, then
	\begin{align*}
		\omega(P_t(x^\ast x)-P_t(x)^\ast P_t(x))&=\omega(F(t)-F(0))\\
		&=\int_0^t \omega(F'(s))\,ds\\
		&=2\int_0^t \omega(P_s\Gamma(P_{t-s}(x)))\,ds\\
		&\geq 2\omega(\Gamma(P_t(x)))\int_0^t e^{2Ks}\,ds\\
		&=e_K(t)\omega(\Gamma(P_t(x))).
	\end{align*}
	This implies the first inequality. The second inequality follows from the fact that $P_t$ is unital and positive.
	
	The last part is easy to see for $K\geq 0$. If $K< 0$ and $0\leq t\leq -\frac 1{2K}$, then
	\begin{equation*}
		e_K(t)=-\frac{1-e^{2K t}}{K}\geq 2(1-e^{-1})t\geq t.\qedhere
	\end{equation*}
	
\end{proof}

\begin{proof}[Proof of Theorem \ref{thm:general_Talagrand}]
	As a consequence of \ref{(A3)},
	\begin{equation*}
		\lVert i_2(x-\phi(x)\un)\rVert^2\leq \frac 1{1-e^{-2\lambda T}}(\lVert i_2(x)\rVert^2-\lVert i_2(P_T(x))\rVert^2).
	\end{equation*}

	Since $(P_t^{(2)})_{t\ge 0}$ is a symmetric contraction semigroup on $L^2(\mathcal M,\phi)$, the map $t\mapsto \lVert i_2(P_t(x))\rVert^2$ is differentiable and
	\begin{align*}
		\lVert i_2(x)\rVert^2-\lVert i_2(P_T(x))\rVert^2&=2\int_0^T \langle i_2(P_t(x)),i_2(\L P_t(x))\rangle\,dt\\
		&=2\int_0^T\sum_{j\in\mathcal J}\lVert i_2(d_j(P_t(x)))\rVert^2\,dt\\
		&=4\int_0^{T/2}\sum_{j\in\mathcal J}\lVert i_2(d_j(P_{2t}(x)))\rVert^2\,dt,
	\end{align*}
	where the second to last step used (\ref{(A2-1)}).

	By \ref{(A5)} and \ref{(A4)},
	\begin{equation*}
		\lVert i_2(d_j(P_{2t}(x)))\rVert\leq e^{-\mu t}\lVert i_2(P_t(d_j(P_t(x))))\rVert\leq e^{-\mu t}\lVert i_{p(t)}(d_j(P_t(x)))\rVert
	\end{equation*}
	with $p(t)=1+e^{-2\alpha t}$.

	By interpolation,
	\begin{equation}\label{eq:interpolation_bound}
		\lVert i_{p(t)}(d_j(P_t(x)))\rVert\leq \lVert i_1(d_j(P_t(x)))\rVert^{1/p(t)}\lVert d_j(P_t(x))\rVert^{1-1/p(t)}.
	\end{equation}
	To the first factor we apply \ref{(A5)} again and the contractivity of $P_t^{(1)}$ to get 
	\begin{equation}\label{eq:bound_L1}
		\lVert i_1(d_j(P_t(x)))\rVert\leq e^{-\mu t}\lVert i_1(P_t(d_j(x)))\rVert\leq e^{-\mu t}\lVert i_1(d_j(x))\rVert.
	\end{equation}

	To bound the second factor in (\ref{eq:interpolation_bound}) in the range $\left[0,\frac 1{2K_-}\right]$, we use Lemma \ref{lem:reverse_Poincare} and (\ref{(A2-2)}) to obtain
	\begin{equation}\label{eq:bound_Linfinity}
		\lVert d_j(P_t(x))\rVert\leq M\lVert \Gamma(P_t(x))\rVert^{1/2}\leq \frac {M}{\sqrt t}.
	\end{equation}

	If we plug (\ref{eq:bound_L1}) and (\ref{eq:bound_Linfinity}) into (\ref{eq:interpolation_bound}), we obtain
	\begin{equation*}
		\lVert i_{p(t)}(d_j(P_t(x)))\rVert\leq M e^{-\mu t/p(t)} t^{-(1-1/p(t))/2}\lVert i_1(d_j(x))\rVert^{1/p(t)}
	\end{equation*}

	for $0\leq t\leq 1/2K_-$. Therefore,
	\begin{align*}
		\lVert i_2(x)\rVert^2-&\lVert i_2(P_T(x))\rVert^2\\
		&\leq 4M\int_0^{T/2} e^{-2\mu(1+1/p(t)) t}t^{-(1-1/p(t))}\sum_{j\in\mathcal J}\lVert i_1(d_j(x))\rVert^{2/p(t)}\,dt\\
		&=4M\sum_{j\in\mathcal J}\lVert i_1(d_j(x))\rVert\int_0^{T/2}e^{-2\mu(1+1/p(t)) t}t^{-(1-1/p(t))}\lVert i_1(d_j(x))\rVert^{2/p(t)-1}\,dt
	\end{align*}
provided $T\leq 1 /K_-$. With $T=\min\{2,1/K_-,1/\alpha\}$, the result follows from Lemma \ref{lem:integral_bound}.
\end{proof}

\begin{proof}[Proof of Theorem \ref{ThmKKLgeneral}]
By Theorem \ref{thm:general_Talagrand} there exists a constant $C>0$ with dependence as indicated in the statement of this theorem such that
\begin{equation*}
1/C\leq \sum_{j\in\mathcal J}\frac{\Inf^1_j(x)(1+\Inf_j(x))}{(1+\log^+(1/\Inf^1_j(x)))^{1/2}}
\end{equation*}
for all $x\in\mathcal A$ with $\lVert i_2(x)\rVert=1$, $\lVert x\rVert\leq 1$ and $\phi(x)=0$.

Now the conclusion follows from Lemma \ref{lem:auxiliary_KKL}.
\end{proof}

\begin{proof}[Proof of Theorem \ref{thm:Friedgut_general_vN}]
The proof follows the same lines as the one of Theorem \ref{thm:L1friedgut}. Let us first consider the case $\Inf^2(x)\geq 1$. Let $\mathcal I=\{j\in\mathcal J\mid\Inf^1_j(x)\leq \eta\}$ for some number $\eta>0$. By the triangle inequality and the contractivity of $E_{\mathcal I}$, we have
\begin{align*}
\lVert i_2(x)-E_{\mathcal I}(i_2(x))\rVert&\leq \lVert i_2(x-P_t(x))\rVert+\lVert (\mathrm{id}-E_{\mathcal I})(i_2(P_t(x)))\rVert+\lVert E_{\mathcal I}(i_2(x-P_t(x)))\rVert\\
&\leq 2\lVert i_2(x-P_t(x))\rVert+\lVert (\mathrm{id}-E_{\mathcal I})(i_2(P_t(x)))\rVert.
\end{align*}
It follows from Lemmas \ref{lem:cor_Poincare} and \ref{lem:elem_spectral_th} that
\begin{equation*}
\lVert i_2(x)-E_{\mathcal I}(i_2(x))\rVert\leq 2\sqrt{t\Inf^2(x)}+\frac{e^{-\mu t}}{\nu}(\eta\Inf^1(x))^{q(t)/2}\Inf^2(x)^{(1-q(t))/2}
\end{equation*}
with $q(t)=\frac{1-e^{-2\alpha t}}{1+e^{-2\alpha t}}$. For 
\begin{equation*}
t=\frac{\epsilon^2}{16\Inf^2(x)},\quad \eta=\frac{\left(\frac{\nu\epsilon}{2e^{\mu_- t}}\right)^{2/q(t)}}{\Inf^1(x)\Inf^2(x)^{(1-q(t))/q(t)}}
\end{equation*}
we obtain $\lVert i_2(x)-E_{\mathcal I}(i_2(x))\rVert\leq \epsilon$.

Moreover, by definition of $\mathcal I$, we have $\Inf^1_j(x)>\eta$ for every $j\in\mathcal J\setminus \mathcal I$. Thus 
\begin{equation*}
\lvert\mathcal J\setminus\mathcal I\rvert\leq \frac{\Inf^1(x)}{\eta}=\left(\frac{2 e^{\mu_- t}}{\nu\epsilon}\right)^{2/q(t)}\Inf^1(x)^2\Inf^2(x)^{(1-q(t))/q(t)}.
\end{equation*} 
The elementary inequality $r\leq e^{r}-1$ for $r\geq 0$ together with $\Inf^2(x)\geq 1$ and $\epsilon\leq 2/\nu$ implies
\begin{equation*}
\frac{2}{q(t)}\leq \frac{e^{2\alpha t}+1}{\alpha t}= \frac{16\Inf^2(x)}{\alpha\epsilon^2}\left(\exp\left(\frac{2\alpha\epsilon^2}{16\Inf^2(x)}\right)+1\right)\leq \frac{16\Inf^2(x)}{\alpha\epsilon^2}\left(\exp\left(\frac{\alpha}{2\nu^2}\right)+1\right).
\end{equation*}
Let $C=\frac{16}{\alpha}\left(\exp\left(\frac{\alpha}{2\nu^2}\right)+1\right)$. Hence
\begin{align*}
\lvert\mathcal J\setminus\mathcal I\rvert&\leq \left(\frac{2 e^{\mu_- t}}{\nu\epsilon}\right)^{C\Inf^2(x)/\epsilon^2}\Inf^1(x)^2\Inf^2(x)^{(1-q(t))/q(t)}\\
&=e^{C\mu_-}\left(\frac{2}{\nu\epsilon}\right)^{C\Inf^2(x)/\epsilon^2}\Inf^1(x)^2\Inf^2(x)^{(1-q(t))/q(t)}.
\end{align*}
Furthermore,
\begin{equation*}
\frac{1-q(t)}{q(t)}\leq \frac 1{\alpha t}=\frac{16\Inf^2(x)}{\alpha \epsilon^2}\leq \frac{C\Inf^2(x)}{\epsilon^2}.
\end{equation*}
Altogether,
\begin{equation*}
\lvert \mathcal J\setminus \mathcal I\rvert\leq\Inf^1(x)^2\exp\left(C\mu_- +\frac{C\Inf^2(x)}{\epsilon^2}\log\frac{2\Inf^2(x)}{\nu\epsilon}\right).
\end{equation*}
The bound in the case $\Inf^2(x)<1$ follows by rescaling.
\end{proof}

\newcommand{\etalchar}[1]{$^{#1}$}

\end{document}